\documentclass[11pt]{article}

\usepackage[a4paper,margin=1in]{geometry}
\usepackage{amsmath,amssymb,amsthm,mathtools}
\usepackage{booktabs}
\usepackage{enumitem}
\usepackage{microtype}
\usepackage{xcolor}
\usepackage[colorlinks=true,linkcolor=blue,citecolor=blue,urlcolor=blue]{hyperref}
\numberwithin{equation}{section}
\newtheorem{theorem}{Theorem}[section]
\newtheorem{proposition}[theorem]{Proposition}
\newtheorem{lemma}[theorem]{Lemma}
\newtheorem{corollary}[theorem]{Corollary}
\newtheorem{remark}[theorem]{Remark}

\newcommand{\T}{\mathbb T}
\newcommand{\R}{\mathbb R}
\newcommand{\F}{\mathcal F}
\newcommand{\U}{\mathcal U}
\newcommand{\qc}{q_{\mathrm c}}
\newcommand{\ac}{\alpha_{\mathrm c}}
\newcommand{\gp}{\gamma'}
\newcommand{\dpd}{d_{\mathrm p}}
\newcommand{\Energy}{\mathcal E}
\newcommand{\dd}{\,\mathrm d}

\title{Maximal Regularity and Existence for Superquadratic Parabolic\\
Hamilton--Jacobi Equations: The Endpoint case}
\author{Fanze Kong \thanks{Department of Applied Mathematics, University of Washington, Seattle, WA 98195, USA; {\sf Email: fzkong@uw.edu}}~ and Xiaoyu Zeng  \thanks{Center for Mathematical Sciences, Wuhan University of Technology, Wuhan 430070, China; xyzeng@whut.edu.cn}}
\date{}

\begin{document}
\maketitle

\begin{abstract}

We establish interior maximal $L^{\qc}$-regularity for bounded
strong solutions of
\[
u_t-\Delta u+|Du|^\gamma=f
\quad\text{in }\T^d\times(0,T),
\]
where $d\ge2$, $\gamma>2$ and
$\qc=(d+2)(\gamma-1)/\gamma$.
The estimates are uniform for uniformly bounded families of
solutions whose source terms range over a bounded, uniformly
equi-integrable subset of $L^{\qc}$. The main difficulty is the possible
concentration of the critical gradient energy. We overcome it
through a two-scale blow-up argument: the first scale produces
a small Hamiltonian coefficient, while energy normalization at
the second scale yields strong endpoint compactness, allowing
a parabolic Liouville theorem to rule out concentration.
The resulting uniform little-H\"older estimate, combined with
critical Gagliardo--Nirenberg interpolation, permits absorption
of the nonlinear term in the parabolic Calder\'on--Zygmund estimate.  As an application, we construct bounded strong solutions for
$f\in L^{\qc}$ and continuous initial data $u_0$, and prove
subsequential strong interior convergence of smooth approximations
in $W^{2,1}_{\qc}$.
 
%We establish the interior maximal \(L^{\qc}\)-regularity for bounded strong
%solutions of
%\[
% u_t-\Delta u+|Du|^\gamma=f
% \quad\text{in }\T^d\times(0,T),
%\]
%where $d\geq 2$, $\gamma>2$ and $\qc=(d+2)(\gamma-1)/\gamma,$ provided that the source terms $f$ range over a bounded, uniformly equi-integrable subset of $L^{\qc}$.  The obstruction to endpoint regularity is the possible concentration of the critical gradient energy.  Using the equi-integrability of the source terms, we establish the uniform little-H\"older modulus through a two-scale blow-up argument
%that rules out concentration of the critical gradient energy. Under the first blow-up, the Hamiltonian term has a small coefficient. If energy concentrates, a second blow-up yields strong endpoint compactness and a nontrivial entire limit, which finally is excluded by a parabolic Liouville theorem.     Combining the resulting uniform little-H\"older modulus with a critical Gagliardo--Nirenberg inequality, we absorb the nonlinear term in the parabolic Calder\'on--Zygmund estimate and prove the maximal regularity of Hamilton--Jacobi equations.   As an application of the maximal regularity result, we establish strong stability at the critical exponent and obtain strong solutions to Hamilton--Jacobi equations.
 
\end{abstract}

\section{Introduction}
The regularity theory of Hamilton--Jacobi equations has a long history.
The study of viscosity solutions to Hamilton--Jacobi equations  was developed in
\cite{CrandallLions1983}, while Bernstein-type arguments for viscous
Hamilton--Jacobi equations can be traced back to \cite{Lions1985}.  H\"older estimates
for coercive or superquadratic Hamiltonians were obtained in
\cite{CannarsaCardaliaguet2010,CapuzzoDolcettaLeoniPorretta2010,
CardaliaguetSilvestre2012,DallAglioPorretta2015} and  Sobolev regularity for
first-order equations was studied in \cite{CardaliaguetPorrettaTonon2015}.  We also mention that adjoint and compensated-compactness methods for
Hamilton--Jacobi equations were developed by Evans \cite{Evans2010}.

For stationary viscous Hamilton--Jacobi equations, the maximal-regularity
problem was raised by P.-L. Lions.  In the periodic setting, Cirant and
Goffi \cite{CirantGoffi2021} proved maximal $L^q$-regularity above the
scaling-critical exponent.  Local superquadratic estimates based on
blow-up and Liouville methods were established by Cirant and Verzini
\cite{CirantVerzini2022}.  Endpoint and maximal regularity phenomena were
further investigated in
\cite{Goffi2023,CirantGoffiLeonori2025,GoffiPediconi2023}, while related
subquadratic regimes and Morrey-space mechanisms arise in
\cite{CirantKongWeiZeng2025}.  For fully nonlinear equations with natural
gradient growth, we also refer to \cite{Goffi2025}.  A recent overview of these developments and the related open problems is
provided in \cite{GoffiSurvey2026}.

The parabolic theory developed along a different route.  Interior H\"older
estimates for equations with superquadratic Hamiltonians and unbounded
 source terms were proved in
\cite{CannarsaCardaliaguet2010,CardaliaguetSilvestre2012}.  Lipschitz
estimates with $L^p$ data were obtained by Cirant and Goffi
\cite{CirantGoffiLipschitz2020}.  The same authors established parabolic
maximal-regularity estimates in \cite{CirantGoffiParabolic2021}.  Cirant
\cite{Cirant2025} subsequently covered the full superquadratic range
strictly above the critical exponent by combining sharp H\"older estimates,
an improvement-of-oscillation argument and a parabolic Liouville theorem.
Related nonlocal and natural-growth problems were considered in
\cite{GoffiNonlocal2021,Goffi2025}. 

Maximal regularity estimates for Hamilton--Jacobi equations have numerous
applications and, in particular, provide important compactness tools for
second-order mean-field game systems, whose mathematical theory was
introduced by Lasry and Lions \cite{LasryLions2007}.  For regularity theory
and variational approaches to mean-field games, we refer, among others, to
\cite{GomesPimentelVoskanyan2016,CirantTonon2019,Munoz2022,
MeszarosSilva2018,CesaroniCirant2019,CirantKongWeiZeng2025,
KongTongZeng2026JDE,KongTongZengZhou2026,
KongTongZeng2026MountainPass}.

We consider the following parabolic Hamilton--Jacobi equation 
\begin{align}\label{eq:main-equation-pub}
  u_t-\Delta u+|Du|^\gamma=f
  \quad\text{a.e. in }\mathbb T^d\times(0,T).
\end{align}
The natural integrability threshold for parabolic maximal
regularity of \eqref{eq:main-equation-pub} arising from the scaling balance is
\begin{align}\label{eq:critical-exponents-pub1}
  q_c:=\frac{d+2}{\gamma'},
  \qquad \gamma':=\frac{\gamma}{\gamma-1},
  \qquad \gamma>1.
\end{align}
At the critical exponent $q=q_c$, we set
\begin{align}\label{eq:critical-exponents-pub2}
  \alpha_c:=2-\frac{d+2}{q_c}
  =\frac{\gamma-2}{\gamma-1}.
\end{align}
These exponents satisfy
\begin{align}\label{eq:critical-exponents-pub3}
  \gamma'q_c=d+2,
  \qquad
  2-\alpha_c=\gamma(1-\alpha_c)=\gamma'.
\end{align}
We show that the critical scaling  balances the three terms in the equation
and preserves the $L^{q_c}$-norm of the source. Indeed, for $z_0=(x_0,t_0)$, define the rescaled  function $u_r(x,t):=r^{-\alpha_c}u(x_0+rx,t_0+r^2t)$.  Under this rescaling, the time derivative, the Laplacian  and the
Hamiltonian term scale with the same power of \(r\).  Moreover, the  rescaled source is $f_{r,z_0}(x,t)
 :=r^{\gamma'}f(x_0+rx,t_0+r^2t) $
whose $L^{q_c}$-norm is invariant since $\gamma'q_c=d+2$. At $q=q_c$, Cirant \cite{Cirant2025} proved the critical parabolic
H\"older estimate in $C^{\alpha_c,\alpha_c/2}$ under uniform equi-integrability of the source terms.
However, maximal regularity in that work requires $q>q_c$.  The endpoint
obstruction can be seen from the critical embedding $W^{2,1}_{q_c}\hookrightarrow W^{1,0}_{\gamma q_c},$
which is continuous but not compact.  Thus weak convergence of a rescaled
sequence does not identify the limit of $|Du_n|^\gamma$.  Equivalently, the
endpoint Gagliardo--Nirenberg inequality produces a term linear in the
$W^{2,1}_{q_c}$ norm, and a merely bounded critical H\"older seminorm does
not provide the small coefficient required for absorption.

To address the endpoint compactness difficulty described above,
we impose a uniform equi-integrability condition on the source terms.
More precisely, let $Q_T:=\T^d\times(0,T)$ and assume that the
admissible source terms form a bounded and uniformly equi-integrable
set $\mathcal F\subset L^{\qc}(Q_T)$; that is,
\[
\sup_{f\in\mathcal F}\|f\|_{L^{\qc}(Q_T)}<\infty
\]
and
\begin{equation}\label{vanishing-omega}
\omega_{\mathcal F}(s)
:=
\sup_{f\in\mathcal F}
\sup_{\substack{E\subset Q_T\\ |E|\le s}}
\int_E |f|^{\qc}\,dx\,dt
\rightarrow0
\qquad\text{as }s\downarrow0.
\end{equation}
For every fixed parabolic cylinder $Q_R$, the identity
$\gamma'\qc=d+2$ and a change of variables give
\[
\|f_{r,z_0}\|_{L^{\qc}(Q_R)}^{\qc}
=
\int_{z_0+\delta_rQ_R}|f(x,t)|^{\qc}\,dx\,dt,
\qquad
\delta_r(x,t):=(rx,r^2t),
\]
whenever $z_0+\delta_rQ_R\subset Q_T$.
Since $|z_0+\delta_rQ_R|=r^{d+2}|Q_R|$, it follows that
\[
\sup_{f\in\mathcal F}
\sup_{\substack{z_0\in Q_T\\ z_0+\delta_rQ_R\subset Q_T}}
\|f_{r,z_0}\|_{L^{\qc}(Q_R)}
\le
\omega_{\mathcal F}\bigl(r^{d+2}|Q_R|\bigr)^{1/\qc}
\rightarrow0
\qquad\text{as }r\downarrow0.
\]
Consequently, for every sequence $r_n\downarrow0$,
$z_n=(x_n,t_n)\in Q_T$, and $f_n\in\mathcal F$, we have
\[
r_n^{\gamma'}
f_n(x_n+r_nx,t_n+r_n^2t)
\rightarrow0
\quad\text{strongly in }L^{\qc}(Q_R),
\]
provided that $z_n+\delta_{r_n}Q_R\subset Q_T$ for all sufficiently
large $n$. This uniform vanishing of the rescaled source terms
is a key ingredient in the proof of our uniform little-H\"older
estimate.

We prove the uniform little-H\"older estimate through a two-scale
blow-up argument and parabolic Liouville rigidity. More precisely,
we adapt the elliptic endpoint argument of~\cite{FanzeKong} to the
parabolic setting, using the critical H\"older estimate and
parabolic Liouville theorem of~\cite{Cirant2025}.
For related blow-up and Liouville methods in elliptic and parabolic
equations, see, among others,
\cite{PeletierSerrin1978,PolacikQuittnerSouplet2007,SoupletZhang2006,
Souplet2023Universal,CirantVerzini2022,Cirant2025}.
The linear parabolic regularity estimates and interpolation
inequalities used below are based on the classical theories
developed in
\cite{LadyzhenskayaSolonnikovUraltseva1968,Lieberman1996,Nirenberg1959,
Nirenberg1966,Triebel1983,Triebel1995}.

Cirant and Goffi~\cite{CirantGoffiParabolic2021} established endpoint
maximal regularity in the subquadratic case, while the superquadratic
result of Cirant~\cite{Cirant2025} requires $q>\qc$.
The following theorem establishes interior maximal regularity
for superquadratic parabolic Hamilton--Jacobi equations at the
critical exponent $q=\qc$.
\begin{theorem}
\label{thm:main-pub}
Suppose that \(d\ge2\), \(\gamma>2\) and 
\(\F\subset L^{\qc}(Q_T)\) is bounded and uniformly equi-integrable in the sense of \eqref{vanishing-omega}.
Let $\U\subset W^{2,1}_{\qc}(Q_T)$ defined in \eqref{W21p-parabolic} be a family of strong
solutions such that each $u\in\U$ satisfies
\eqref{eq:main-equation-pub} with a source $f_u\in\F$.
%\begin{equation}\label{eq:main-equation-pub}
% u\in W^{2,1}_{\qc}(Q_T),
% \qquad
% u_t-\Delta u+|Du|^\gamma=f_u
% \quad\text{a.e. in }Q_T.
%\end{equation}
Assume  that
\begin{equation}\label{eq:uniform-Linfty-pub}
 \sup_{u\in\U}\|u\|_{L^\infty(Q_T)}\le K.
\end{equation}
Then, for every \(0<\tau<T/2\), there exists a positive constant $C_\tau=C_\tau\!\left(d,\gamma,T,\tau,K,
\sup_{f\in\F}\|f\|_{L^{\qc}(Q_T)},\omega_{\F}\right)$
such that
\begin{equation}\label{eq:main-estimate-pub}
 \sup_{u\in\U}
 \left(
 \|u_t\|_{L^{\qc}(Q_\tau)}
 +\|D^2u\|_{L^{\qc}(Q_\tau)}
 +\bigl\||Du|^\gamma\bigr\|_{L^{\qc}(Q_\tau)}
 \right)
 \le C_\tau,
\end{equation}
where $Q_\tau:=\T^d\times(\tau,T-\tau)$.
 
\end{theorem}

To prove Theorem~\ref{thm:main-pub}, we first establish a uniform
little-H\"older estimate by means of a two-stage blow-up argument and
parabolic Liouville rigidity.  The first blow-up is taken around points
and scales at which this estimate fails, producing a small coefficient
in front of the Hamiltonian.  If the resulting sequence has unbounded
critical energy, we introduce a second, energy-normalized rescaling.  At
the second scale, the linearized drift is small in $L^{d+2}$, which
upgrades weak convergence to strong
$W^{2,1}_{q_c,\mathrm{loc}}$-convergence.  Liouville rigidity then
contradicts the positive normalized energy and yields a universal local
critical-energy bound.  Returning to the first scale, this bound,
together with the small Hamiltonian coefficient, gives strong local
compactness, thereby eliminating the nonlinear defect and proving the
uniform little-H\"older estimate.  Finally, the endpoint
Gagliardo--Nirenberg inequality and the standard
Calder\'on--Zygmund estimate yield the maximal-regularity estimate of \eqref{eq:main-equation-pub}.   This result, together with the underlying blow-up framework, is the
parabolic counterpart of the elliptic theory developed in
\cite{FanzeKong}.  We emphasize the precise role of the second rescaling.  For the 
sequence produced by the first blow-up, the coefficient of the Hamiltonian
is the fixed small number
\(\lambda=\varepsilon^{\gamma-1}\). In this case, the local
parabolic estimate and the endpoint interpolation inequality allow the
nonlinear term to be absorbed into the left-hand side. Proposition
\ref{prop:universal-energy-pub}, however, gives a stronger result: it provides
a uniform local critical-energy bound for every admissible sequence and
every \(0\leq\lambda_n\leq1\), with a constant independent of the coefficient
sequence. Such a coefficient-uniform estimate cannot be obtained by the
preceding absorption argument when \(\lambda_n\) is not small. The second,
energy-normalized blow-up is introduced precisely to rule out critical-energy
concentration in this general setting.

As an application of our maximal-regularity result, we establish
the existence of bounded strong solutions for $f\in L^{\qc}(Q_T)$
and $u_0\in C(\T^d)$. The proof also requires control near the
initial time, since interior estimates alone do not ensure that
the limit solution attains the prescribed initial data. We obtain
a uniform modulus of continuity at $t=0$ by combining the adjoint
identity and the endpoint adjoint estimate
inspired by ~\cite{CirantGoffiParabolic2021} with uniform tail estimates
for the source terms and a distance estimate. Together with
the interior compactness argument, this yields the existence
and strong convergence results stated below.  As an application of Theorem~\ref{thm:main-pub} and
Proposition~\ref{prop:universal-energy-pub}, we obtain the following
existence result for \eqref{eq:main-equation-pub}:
\begin{theorem}
\label{thm:critical-existence-pub}
{Let $d\ge2$ and $\gamma>2$.} Let $f\in L^{\qc}(Q_T)$ and
$u_0\in C(\T^d)$. Then there exists $u\in W^{2,1}_{\qc,\mathrm{loc}}(Q_T)
\cap C(\T^d\times[0,T)){\cap L^\infty(Q_T)}$
which is a strong solution of
\begin{equation}\label{eq:critical-limit-equation-pub}
 \begin{cases}
 u_t-\Delta u+|Du|^\gamma=f&\text{a.e. in }Q_T,\\
 u(\cdot,0)=u_0&\text{on }\T^d.
 \end{cases}
\end{equation}
More precisely, $\lim_{t\downarrow0}
 \|u(\cdot,t)-u_0\|_{L^\infty(\mathbb T^d)}=0.$
\end{theorem}
In the proof of Theorem \ref{thm:critical-existence-pub}, as a corollary, we have the following stability result: 
\begin{corollary}
Let $f_n\in {C^\infty(\T^d\times[0,T])}$ satisfy $f_n\rightarrow f$ \text{strongly in }$L^{q_c}(Q_T)$, and let $u_{0,n}:=\rho_{1/n}*u_0$,
where $(\rho_\varepsilon)_{\varepsilon>0}$ is a family of standard
periodic mollifiers. For each $n$, let $u_n$ be the classical solution of
\[
 \begin{cases}
 (u_n)_t-\Delta u_n+|Du_n|^\gamma=f_n
     &\text{in }Q_T,\\
 u_n(\cdot,0)=u_{0,n}
     &\text{on }\mathbb T^d.
 \end{cases}
\]
Then there exist a subsequence, still denoted by $(u_n)$, and a solution
$u$ provided by Theorem~\ref{thm:critical-existence-pub} such that, for
every $0<\tau<T/2$,
\begin{equation*}%\label{eq:critical-strong-stability-pub}
 \begin{aligned}
 &\|u_n-u\|_{W^{2,1}_{q_c}(Q_\tau)}
 +\|Du_n-Du\|_{L^{\gamma q_c}(Q_\tau)}
 +\bigl\||Du_n|^\gamma-|Du|^\gamma
   \bigr\|_{L^{q_c}(Q_\tau)}
 \rightarrow0.
 \end{aligned}
\end{equation*}
Moreover, $u_n\rightarrow u$ locally uniformly in $Q_T$.

\end{corollary}

The remainder of the paper is organized as follows.  In
Section~\ref{sect2}, we collect preliminary results, including the endpoint
interpolation inequality, small-drift regularity estimates and a Liouville
theorem.  Section~\ref{sec:first-scale-pub} is devoted to the two-stage
blow-up argument, through which we establish a universal local
critical-energy bound, prove strong local compactness, and derive the
uniform little-H\"older estimate.  Finally, in
Section~\ref{sec:absorption-pub}, the endpoint interpolation inequality and
the parabolic Calder\'on--Zygmund estimate are combined to complete the
proof of Theorem~\ref{thm:main-pub}.  In Section~{\ref{sect5}}, we apply the
coefficient-uniform energy estimate shown in Proposition \ref{prop:universal-energy-pub} and Theorem \ref{thm:main-pub} to prove the strong stability and
existence result stated in Theorem~\ref{thm:critical-existence-pub}.
 
\section{Notation and preliminaries}\label{sect2}
In this section, we introduce the notation and establish the preliminary
results needed in the sequel. 
 For \(z=(x,t)\) and \(\bar z=(\bar x,\bar t)\), we  set the distance
\begin{equation*}%\label{eq:parabolic-distance-pub}
 \dpd(z,\bar z):=
 \operatorname{dist}_{\T^d}(x,\bar x)+|t-\bar t|^{1/2}.
\end{equation*}
In addition, for 
\(z_0=(x_0,t_0)\), we denote
\begin{equation*}%\label{eq:parabolic-cylinder-pub}
 Q_r(z_0):=B_r(x_0)\times(t_0-r^2,t_0+r^2),
 \qquad Q_r:=Q_r(0,0),
\end{equation*}
where $r>0$ is a constant.  %Define $ \gp:=\frac{\gamma}{\gamma-1}$, then
%\begin{equation}\label{eq:critical-exponents-pub}
% \qc:=\frac{d+2}{\gp}
  %    =\frac{(d+2)(\gamma-1)}{\gamma},
 %\qquad
 %\ac:=2-\frac{d+2}{\qc}
%     =\frac{\gamma-2}{\gamma-1}.
%\end{equation}
%The identities used below are
%\begin{equation}\label{eq:critical-identities-pub}
% 0<\ac<1,
% \quad 1<\qc<d+2,
% \quad \gp\qc=d+2,
% \quad \gamma\qc=(d+2)(\gamma-1),
%\end{equation}
%and
%\begin{equation}\label{eq:homogeneity-identities-pub}
 %2-\ac=\gp,
% \qquad \gamma(1-\ac)=\gp.
%\end{equation}
%Define $\mathcal F$ is the uniformly equi-integrable subset of % $L^{\qc}(Q_T)$.  More precisely, set
%\begin{equation}\label{eq:equi-modulus-pub}
% \omega_{\F}(s):=
% \sup_{f\in\F}\ \sup_{\substack{E\subset Q_T\\ |E|\le s}}
% \int_E |f|^{\qc}\dd x\dd t,
%\end{equation}
%then we say \(\F\) is bounded and uniformly equi-integrable in
%\(L^{\qc}(Q_T)\) if
%\begin{align}\label{vanishing-omega}
% \sup_{f\in\F}\|f\|_{L^{\qc}(Q_T)}<\infty
% \quad\text{and}\quad
% \omega_{\F}(s)\rightarrow0\quad\text{ as }s\downarrow0.
%\end{align}
For an open set \(G\subset\mathbb R^d\times\mathbb R\) and
\(1\le p<\infty\), define
\begin{align}\label{W21p-parabolic}
W^{2,1}_p(G)
:=
\left\{
u\in L^p(G):
u_t,\,Du,\,D^2u\in L^p(G)
\right\},
\end{align}
equipped with the equipped  norm 
\[
\|u\|_{W^{2,1}_p(G)}
:=
\|u\|_{L^p(G)}
+\|u_t\|_{L^p(G)}
+\|Du\|_{L^p(G)}
+\|D^2u\|_{L^p(G)}.
\]
We first recall  the  standard Calder\'on--Zygmund estimate for the heat equation.
\begin{lemma}[Chapter IV in \cite{LadyzhenskayaSolonnikovUraltseva1968}]
\label{lem:preliminary-linear-estimates}
Let \(1<p<\infty\) and \(0<r<R\le2\). Then there exists a constant
\(C=C(d,p)>0\)  such that every
\(z\in W^{2,1}_p(Q_R)\) satisfies
\begin{equation}\label{eq:local-CZ-pub}
\|z_t\|_{L^p(Q_r)}
+\|D^2z\|_{L^p(Q_r)}
\le
C\left(
\|z_t-\Delta z\|_{L^p(Q_R)}
+(R-r)^{-2}\|z\|_{L^p(Q_R)}
\right).
\end{equation}
If, in addition, \(1<p<d+2\), define \(p^\sharp\) by
\begin{equation}\label{eq:p-sharp-pub}
\frac1{p^\sharp}
:=
\frac1p-\frac1{d+2},
\end{equation}
where $p^\sharp=\frac{(d+2)p}{d+2-p}.$
Then
\begin{equation}\label{eq:parabolic-sobolev-pub}
\|Dz\|_{L^{p^\sharp}(Q_r)}
\le
C\left(
\|z_t\|_{L^p(Q_R)}
+\|D^2z\|_{L^p(Q_R)}
+(R-r)^{-2}\|z\|_{L^p(Q_R)}
\right).
\end{equation}
\end{lemma}

\begin{proof}
 \eqref{eq:local-CZ-pub} is the standard interior
Calder\'on--Zygmund estimate of heat equations and we refer the reader to \cite{LadyzhenskayaSolonnikovUraltseva1968}.  

To prove \eqref{eq:parabolic-sobolev-pub}, we set $\delta:=R-r$
and choose \(\eta\in C_c^\infty(Q_R)\) such that $0\le\eta\le1,~
\eta\equiv1\ \text{on }Q_r$ and
\[
|D\eta|\le C\delta^{-1},
\qquad
|D^2\eta|+|\eta_t|\le C\delta^{-2}.
\]
Applying the Sobolev inequality to \(\eta z\), extended
by zero outside \(Q_R\), gives
\[
\|D(\eta z)\|_{L^{p^\sharp}(\mathbb R^d\times\mathbb R)}
\le
C\|\eta z\|_{W^{2,1}_p(\mathbb R^d\times\mathbb R)}.
\]
Hence,
\begin{align*}
\|Dz\|_{L^{p^\sharp}(Q_r)}
&\le
C\Big(
\|z_t\|_{L^p(Q_R)}
+\|D^2z\|_{L^p(Q_R)}
+\delta^{-1}\|Dz\|_{L^p(Q_R)}
+\delta^{-2}\|z\|_{L^p(Q_R)}
\Big).
\end{align*}
In addition, the Sobolev interpolation inequality implies
\[
\delta^{-1}\|Dz\|_{L^p(Q_R)}
\le
C\|D^2z\|_{L^p(Q_R)}
+C\delta^{-2}\|z\|_{L^p(Q_R)}.
\]
Combining the inequalities shown above, we  prove
\eqref{eq:parabolic-sobolev-pub}.
\end{proof}

We now state the key {small-drift regularity lemma} for the linearized
Hamilton--Jacobi equation with a small drift.
\begin{lemma}
\label{lem:small-drift-pub}
Let \(1<p<d+2\).  There are constants
\(\kappa_p>0\) and \(C_p>0\), depending only on \(d\) and \(p\)
such that the following properties hold.

\begin{enumerate}[label=\textup{(\roman*)}]
\item If \(\mathbf{b}\in L^{d+2}(Q_2;\R^d)\),
\(\|\mathbf{b}\|_{L^{d+2}(Q_2)}\le\kappa_p\), and
\(z\in W^{2,1}_p(Q_2)\) satisfies
\begin{equation}\label{eq:drift-equation-pub}
 z_t-\Delta z+\mathbf{b}\cdot Dz=h
 \quad\text{in }Q_2,
\end{equation}
then
\begin{equation}\label{eq:small-drift-estimate-pub}
 \|z_t\|_{L^p(Q_1)}+\|D^2z\|_{L^p(Q_1)}
 \le C_p\left(
 \|h\|_{L^p(Q_2)}+\|z\|_{L^p(Q_2)}
 \right).
\end{equation}

\item Let \(1<p_0<p<d+2\).  Suppose
\(z\in W^{2,1}_{p_0}(Q_2)\) solves
\eqref{eq:drift-equation-pub}, \(h\in L^p(Q_2)\), and $ \|\mathbf{b}\|_{L^{d+2}(Q_2)}
 \le\min\{\kappa_{p_0},\kappa_p\}.$
If, in addition,
\begin{equation}\label{eq:drift-upgrade-assumptions-pub}
 z\in L^\infty_{\mathrm{loc}}(Q_2),
 \qquad
 Dz\in L^s_{\mathrm{loc}}(Q_2)
 \quad\text{for some }s>d+2,
\end{equation}
then  \(z\in W^{2,1}_p(Q_1)\).
\end{enumerate}
%\rev{The constants $\kappa_p$ are decreased, if necessary, so that the
%smallness thresholds apply both to the interior estimates and to the
%maximal-regularity estimates with zero forward parabolic boundary data
%used below.}
\end{lemma}

\begin{proof}
We first prove conclusion (i).  Fix \(1<r<R<2\) and set
\[
 X(\rho):=
 \|z_t\|_{L^p(Q_\rho)}+\|D^2z\|_{L^p(Q_\rho)}.
\]
Let \(m=(r+R)/2\).  Applying \eqref{eq:local-CZ-pub} in Lemma \ref{lem:preliminary-linear-estimates} on
\(Q_r\Subset Q_m\) to
\(z_t-\Delta z=h-\mathbf{b}\cdot Dz\), we obtain
\begin{equation}\label{eq:drift-CZ-step-pub}
 X(r)\le C_p\left(
 \|h\|_{L^p(Q_m)}+\|\mathbf{b}\cdot Dz\|_{L^p(Q_m)}
 +(R-r)^{-2}\|z\|_{L^p(Q_m)}
 \right).
\end{equation}
Using H\"older's inequality and \eqref{eq:p-sharp-pub}, one has
\[
 \|\mathbf{b}\cdot Dz\|_{L^p(Q_m)}
 \le\|\mathbf{b}\|_{L^{d+2}(Q_m)}
      \|Dz\|_{L^{p^\sharp}(Q_m)}.
\]
Applying the estimate shown in  \eqref{eq:parabolic-sobolev-pub} on
\(Q_m\Subset Q_R\), we obtain
\begin{equation}\label{eq:drift-product-pub}
 \|\mathbf{b}\cdot Dz\|_{L^p(Q_m)}
 \le C_p\|\mathbf{b}\|_{L^{d+2}(Q_2)}
 \left(
 X(R)+(R-r)^{-2}\|z\|_{L^p(Q_R)}
 \right).
\end{equation}
Combining \eqref{eq:drift-CZ-step-pub} and
\eqref{eq:drift-product-pub} yields
\begin{equation}\label{eq:drift-hole-pub}
 X(r)
 \le C_p\|\mathbf{b}\|_{L^{d+2}(Q_2)}X(R)
 +C_p(R-r)^{-2}
 \left(
 \|h\|_{L^p(Q_2)}+\|z\|_{L^p(Q_2)}
 \right).
\end{equation}
Choose \(\kappa_p\) so that \(C_p\kappa_p\le1/8\).  Set
\(r_j=2-2^{-j}\), \(j\ge0\).  Iterating
\eqref{eq:drift-hole-pub} along \((r_j)_j\)  with
\(
 A=\|h\|_{L^p(Q_2)}+\|z\|_{L^p(Q_2)}
\)
gives
\[
 X(1)\le8^{-m}X(r_m)
 +C_pA\sum_{j=0}^{m-1}8^{-j}4^{j+1}.
\]
Letting $m\rightarrow+\infty$, we finish the proof of
\eqref{eq:small-drift-estimate-pub}.

We now prove conclusion (ii).  Fix \(1<r<R<2\), and choose
\(\chi\in C_c^\infty(Q_R)\) with \(\chi=1\) on \(Q_r\).  Define
\(Z=\chi z\).  A direct computation  gives
\begin{equation*}%\label{eq:localized-drift-pub}
 Z_t-\Delta Z+\mathbf{b}\cdot DZ
 =\chi h+\chi_tz-2D\chi\cdot Dz-z\Delta\chi
  +(\mathbf{b}\cdot D\chi)z=:F.
\end{equation*}
The assumptions in \eqref{eq:drift-upgrade-assumptions-pub} imply
\(F\in L^p(Q_R)\).  Indeed, since \(p<d+2<s\), $
 D\chi\cdot Dz\in L^s(Q_R)\subset L^p(Q_R),$
 and $(\mathbf{b}\cdot D\chi)z\in L^{d+2}(Q_R)\subset L^p(Q_R).$ 
 
 Let
$\partial_{\mathrm p}Q_R
 :=\bigl(\partial B_R\times[-R^2,R^2]\bigr)
   \cup\bigl(B_R\times\{-R^2\}\bigr)$
be the forward parabolic boundary and set $X_p(Q_R):=
 \{Y\in W^{2,1}_p(Q_R):Y=0
 \text{ on }\partial_{\mathrm p}Q_R\}.$  For \(0\le\theta\le1\), define
\(L_\theta Y=Y_t-\Delta Y+\theta \mathbf{b}\cdot DY\).  By using $W^{2,1}_p$
estimate, H\"older's inequality and the   Sobolev inequality, we have 
\begin{align}\label{RHS}
 \|Y\|_{W^{2,1}_p(Q_R)}
 \le C_p\|L_\theta Y\|_{L^p(Q_R)}
 +C_p\|\mathbf{b}\|_{L^{d+2}(Q_R)}
       \|Y\|_{W^{2,1}_p(Q_R)}.
\end{align}
The last term in (\ref{RHS}) is absorbed by the fact that \(\|\mathbf{b}\|_{L^{d+2}(Q_2)}\le\kappa_p\).  Since \(L_0:X_p(Q_R)\to L^p(Q_R)\) is an isomorphism,
the method of continuity shows that \(L_1\) is an isomorphism.  Hence,
there is a unique \(Y\in X_p(Q_R)\) satisfying \(L_1Y=F\).

Since \(p>p_0\) and \(Q_R\) has finite measure,
\(Y,Z\in W^{2,1}_{p_0}(Q_R)\).  Their subtraction satisfies
\(L_1(Y-Z)=0\) with zero boundary condition.  The invertibility estimate further implies \(Y=Z\).  Therefore,
\(Z\in W^{2,1}_p(Q_R)\)  and \(z=Z\in W^{2,1}_p(Q_r)\).  Letting
\(r\downarrow1\) it finishes the  proof of  (ii).
\end{proof}

We next establish Liouville theorem of the homogeneous Hamilton--Jacobi equation.
\begin{lemma}
\label{lem:endpoint-liouville-pub}
Let \(\mu\ge0\), and suppose $v\in W^{2,1}_{\qc,\mathrm{loc}}(\R^d\times\R)
 \cap C^{\ac,\ac/2}(\R^d\times\R)$
solves
\begin{equation*}%\label{eq:liouville-equation-pub}
 v_t-\Delta v+\mu|Dv|^\gamma=0
 \quad\text{in }\R^d\times\R,
\end{equation*}
and satisfies
\begin{equation*}%\label{eq:liouville-growth-pub}
 \sup_{z\ne\bar z}
 \frac{|v(z)-v(\bar z)|}{\dpd(z,\bar z)^{\ac}}<\infty.
\end{equation*}
Then \(v\) is constant.
\end{lemma}

\begin{proof}
If \(\mu=0\), the conclusion is standard, see for example \cite{SoupletZhang2006}. When \(\mu>0\), we rewrite the equation as
\begin{equation*}%\label{eq:liouville-linearized-pub}
 v_t-\Delta v+\mathbf{b}\cdot Dv=0,
 \qquad
 \mathbf{b}:=\mu|Dv|^{\gamma-2}Dv.
\end{equation*}
Invoking the Sobolev embedding, one has
\begin{equation}\label{eq:critical-gradient-embedding-pub}
 Dv\in L^{\gamma\qc}_{\mathrm{loc}},
 \qquad \gamma\qc=(d+2)(\gamma-1).
\end{equation}
It then follows that  \(\mathbf{b}\in L^{d+2}_{\mathrm{loc}}\).  Fix \(z_0\), we   use the
absolute continuity of the integral to obtain a constant $r>0$ such that  
\[
 \|\mathbf{b}\|_{L^{d+2}(Q_{2r}(z_0))}
 \le\min\{\kappa_{\qc},\kappa_p\}
\]
for some fixed \(p\) satisfying \(\qc<p<d+2\).  Under the parabolic standard rescaling, the drift becomes
\(\mathbf{b}_r(y,s)=r \mathbf{b}(x_0+ry,t_0+r^2s)\) and
\(
 \|\mathbf{b}_r\|_{L^{d+2}(Q_2)}
 =\|\mathbf{b}\|_{L^{d+2}(Q_{2r}(z_0))}.
\)
The H\"older estimate gives local boundedness of \(v\), while
\eqref{eq:critical-gradient-embedding-pub} gives $ Dv\in L^{(d+2)(\gamma-1)}_{\mathrm{loc}}$ with
$(d+2)(\gamma-1)>d+2.$
Now, we apply conclusion (ii) in Lemma~\ref{lem:small-drift-pub} with \(p_0=\qc\) and \(h=0\), then 
obtain \(v\in W^{2,1}_{p,\mathrm{loc}}\).  Fix any
\(T_0\in\R\) and set \(\widetilde v(x,s)=v(x,T_0-s)\) for \(s>0\).
Moreover, applying \cite[Theorem 4.1]{Cirant2025} to \(\widetilde v\), we conclude
that \(\widetilde v\) is constant on
\(\mathbb R^d\times(0,\infty)\). Equivalently, \(v\) is constant on
\(\mathbb R^d\times(-\infty,T_0)\). Since
\(T_0\) is arbitrary, \(v\) is constant in $\mathbb R^d\times \mathbb R.$ 
\end{proof}

We next state the Gagliardo--Nirenberg inequality used below.
\begin{lemma}
\label{lem:GN-pub}
If \(B=B_r(x_0)\subset\R^d\), \(r>0\) and
\(v\in W^{2,\qc}(B)\cap C^{0,\ac}(B)\), then
\begin{equation}\label{eq:GN-pub}
 \|Dv\|_{L^{\gamma\qc}(B)}^\gamma
 \le C[v]_{C^{0,\ac}(B)}^{\gamma-1}
       \|D^2v\|_{L^{\qc}(B)}
      +Cr^{-2/\qc}[v]_{C^{0,\ac}(B)}^\gamma,
\end{equation}
where constant \(C=C(d,\gamma)>0\).
\end{lemma}

\begin{proof}
By using Gagliardo--Nirenberg interpolation
inequality \cite{Nirenberg1966},  we obtain 
\begin{equation}\label{eq:GN-base-pub}
 \|Dv\|_{L^{\gamma\qc}(B_1)}
 \le C\|D^2v\|_{L^{\qc}(B_1)}^\theta
       [v]_{C^{0,\ac}(B_1)}^{1-\theta}
      +C[v]_{C^{0,\ac}(B_1)},
\end{equation}
where \(\theta\) is determined by
\[
 \frac1{\gamma\qc}-\frac1d
 =\theta\left(\frac1{\qc}-\frac2d\right)
  -(1-\theta)\frac{\ac}{d}.
\]
Using \eqref{eq:critical-exponents-pub1}, \eqref{eq:critical-exponents-pub2} and \eqref{eq:critical-exponents-pub3}, one obtains
\(\theta=1/\gamma\).  Raising \eqref{eq:GN-base-pub} to the power
\(\gamma\) proves the desired estimate in $B_1$.

For $B_r$, set \(w(y)=v(x_0+ry)\) on \(B_1\).  Then $\|Dw\|_{L^{\gamma\qc}(B_1)}^\gamma
 =r^{\gamma-d/\qc}\|Dv\|_{L^{\gamma\qc}(B_r)}^\gamma, $
and
\[
 [w]_{C^{0,\ac}(B_1)}=r^\ac[v]_{C^{0,\ac}(B_r)},
 \qquad
 \|D^2w\|_{L^{\qc}(B_1)}
 =r^{2-d/\qc}\|D^2v\|_{L^{\qc}(B_r)}.
\]
%The identity \(\ac(\gamma-1)+2=\gamma\) shows that the first term is
%scale invariant.  For the lower-order term, using
%\(\gamma(1-\ac)=\gp=(d+2)/\qc\), we obtain
%\[
% r^{\ac\gamma-(\gamma-d/\qc)}=r^{-2/\qc}.
%\]
{Apply the $\gamma$-th power of \eqref{eq:GN-base-pub} to $w$,
substitute these scaling identities and divide by $r^{\gamma-d/\qc}$.
Since $\alpha_c(\gamma-1)+2=\gamma$ and
$\alpha_c\gamma-\gamma+d/\qc=-2/\qc$, this yields \eqref{eq:GN-pub}.}

\end{proof}

As a consequence of Lemma~\ref{lem:GN-pub}, we obtain the following
corollary.
\begin{lemma}
\label{lem:integrated-GN-pub}
Let \(I\Subset(0,T)\), \(u\in W^{2,1}_{\qc}(\T^d\times I)\), and
suppose that, for some \(r_0>0\) and \(\sigma>0\),
\begin{equation}\label{eq:small-spatial-holder-pub}
 \operatorname*{ess\,sup}_{t\in I}
 \sup_{0<\operatorname{dist}(x,y)\le4r_0}
 \frac{|u(x,t)-u(y,t)|}{\operatorname{dist}(x,y)^{\ac}}
 \le\sigma.
\end{equation}
Assume   that \(4r_0\) is smaller than the
injectivity radius of \(\T^d\).  Then
\begin{equation}\label{eq:integrated-GN-precise-pub}
 \bigl\||Du|^\gamma\bigr\|_{L^{\qc}(\T^d\times I)}
 \le C\sigma^{\gamma-1}
       \|D^2u\|_{L^{\qc}(\T^d\times I)}
      +C T^{1/\qc}r_0^{-\gp}\sigma^\gamma,
\end{equation}
where \(C=C(d,\gamma)\). 
\end{lemma}

\begin{proof}
By Fubini's theorem, for almost every \(t\in I\) one has
\(u(\cdot,t)\in W^{2,\qc}(\T^d)\), and
\eqref{eq:small-spatial-holder-pub} holds.  Cover
\(\T^d\) by   balls \(B_j=B_{r_0}(x_j)\),
\(j=1,\ldots,N\), such that
\begin{equation}\label{eq:torus-cover-pub}
 \T^d=\bigcup_{j=1}^N B_j,
 \qquad
 \sum_{j=1}^N\mathbf1_{2B_j}\le C_d,
 \qquad
 N\le C_dr_0^{-d}.
\end{equation}
%Because \(4r_0\) is below the injectivity radius, every \(2B_j\) is
%identified isometrically with a Euclidean ball of radius \(2r_0\).
Then, we have any two points of \(2B_j\) are at distance at most \(4r_0\);
hence \eqref{eq:small-spatial-holder-pub} gives
\begin{equation}\label{eq:slice-holder-pub}
 [u(\cdot,t)]_{C^{0,\ac}(2B_j)}\le\sigma.
\end{equation}

Apply Lemma~\ref{lem:GN-pub} on \(2B_j\), we obtain from  \eqref{eq:slice-holder-pub} that
\begin{equation*}%\label{eq:local-slice-GN-pub}
 \|Du(\cdot,t)\|_{L^{\gamma\qc}(2B_j)}^\gamma
 \le C\sigma^{\gamma-1}
       \|D^2u(\cdot,t)\|_{L^{\qc}(2B_j)}
      +Cr_0^{-2/\qc}\sigma^\gamma.
\end{equation*}
Since \(\qc>1\), raising this inequality to the power \(\qc\)   gives
\begin{align}
 \int_{2B_j}|Du(x,t)|^{\gamma\qc}\dd x
 &\le C\sigma^{(\gamma-1)\qc}
       \int_{2B_j}|D^2u(x,t)|^{\qc}\dd x
      +Cr_0^{-2}\sigma^{\gamma\qc}.
 \label{eq:local-slice-powered-pub}
\end{align}
Since the balls \(B_j\) cover \(\T^d\), summing
\eqref{eq:local-slice-powered-pub} and using
\eqref{eq:torus-cover-pub} gives, for almost every \(t\in I\),
\begin{align}
 \int_{\T^d}|Du(x,t)|^{\gamma\qc}\dd x
 &\le
 \sum_{j=1}^N\int_{2B_j}|Du(x,t)|^{\gamma\qc}\dd x
 \notag\\
 &\le C\sigma^{(\gamma-1)\qc}
       \int_{\T^d}|D^2u(x,t)|^{\qc}\dd x
      +Cr_0^{-(d+2)}\sigma^{\gamma\qc}.
 \label{eq:global-slice-powered-pub}
\end{align}
Integrating \eqref{eq:global-slice-powered-pub} over \(I\), we further obtain
\begin{align}
 \int_I\int_{\T^d}|Du|^{\gamma\qc}\dd x\dd t
 &\le C\sigma^{(\gamma-1)\qc}
       \|D^2u\|_{L^{\qc}(\T^d\times I)}^{\qc}
      +CTr_0^{-(d+2)}\sigma^{\gamma\qc}.
 \label{eq:space-time-powered-pub}
\end{align}
 In light of \((d+2)/\qc=\gp\), inequality
\eqref{eq:space-time-powered-pub} implies
\eqref{eq:integrated-GN-precise-pub}.  
\end{proof}

We now state H\"older estimate of (\ref{eq:main-equation-pub}) and remark that it is insufficient to prove the maximal $L^{q_c}$ regularity.
\begin{theorem}[Theorem 3.1 in \cite{Cirant2025}]
\label{thm:critical-holder-pub}
Let \(G\Subset Q_T\).  Under the hypotheses of
Theorem~\ref{thm:main-pub}, there is a constant \(H_G\), independent of
\(u\in\U\), such that
\begin{equation*}%\label{eq:critical-holder-pub}
 [u]_{C^{\ac,\ac/2}(G)}
 :=\sup_{\substack{z,\bar z\in G\\z\ne\bar z}}
 \frac{|u(z)-u(\bar z)|}{\dpd(z,\bar z)^{\ac}}
 \le H_G.
\end{equation*}
\end{theorem}

%This is the interior version of \cite[Theorem 3.1]{Cirant2025} at
%\(q=(d+2)/\gamma'\), after reversing time.  The theorem in
%\cite{Cirant2025} is stated with separate weighted spatial and temporal
%seminorms.  On \(G\Subset Q_T\), those weights are bounded away from zero,
%and
%\[
 %|u(x,t)-u(y,s)|
 %\le |u(x,t)-u(y,t)|+|u(y,t)-u(y,s)|
%\]
%gives \eqref{eq:critical-holder-pub}.

Moreover, we prove little-H\"older estimate of function $u\in W^{2,1}_{q_c}$.
\begin{lemma}
\label{lem:individual-little-pub}
Let \(G\Subset Q_T\) and \(u\in W^{2,1}_{\qc}(Q_T)\).  Then
\begin{equation*}%\label{eq:individual-little-pub}
 \lim_{r\downarrow0}
 \sup_{\substack{z,\bar z\in G\\0<\dpd(z,\bar z)\le r}}
 \frac{|u(z)-u(\bar z)|}{\dpd(z,\bar z)^{\ac}}=0.
\end{equation*}
\end{lemma}

\begin{proof}
Choose a smooth open set \(G'\) with \(G\Subset G'\Subset Q_T\).  Since
\(\qc>(d+2)/2\) and \(\ac=2-(d+2)/\qc\in(0,1)\), the {continuous parabolic Morrey embedding} is
\begin{equation}\label{eq:morrey-pub}
 W^{2,1}_{\qc}(G')\hookrightarrow C^{\ac,\ac/2}(G).
\end{equation}
Fix \(\varepsilon>0\), by the density property, we choose
\(\phi\in C^\infty(G')\) such that
\[
 \|u-\phi\|_{W^{2,1}_{\qc}(G')}
 <\frac{\varepsilon}{2C_M},
\]
where \(C_M\) is a constant obtained in \eqref{eq:morrey-pub}.
Thus,
\begin{equation}\label{eq:smooth-approx-holder-pub}
 [u-\phi]_{C^{\ac,\ac/2}(G)}<\varepsilon/2.
\end{equation}
For \(z=(x,t),\bar z=(y,s)\in G\), the smooth approximation implies 
\begin{align}\label{divideby}
 |\phi(z)-\phi(\bar z)|
 \le\|D\phi\|_\infty|x-y|+\|\phi_t\|_\infty|t-s|.
\end{align}
When \(\dpd(z,\bar z)\le r\le1\), division (\ref{divideby}) by
\(\dpd(z,\bar z)^{\ac}\) yields
\[
 \frac{|\phi(z)-\phi(\bar z)|}{\dpd(z,\bar z)^{\ac}}
 \le\|D\phi\|_\infty r^{1-\ac}
     +\|\phi_t\|_\infty r^{2-\ac}.
\]
Choosing \(r\) so that the right-hand side above is less than
\(\varepsilon/2\), and combining this with
\eqref{eq:smooth-approx-holder-pub} gives us the desired vanishing property.
\end{proof}

\section{Uniform little-H\"older estimate }
\label{sec:first-scale-pub}
In this section, we establish a uniform little-H\"older estimate for the
solutions of \eqref{eq:main-equation-pub} by using a blow-up argument.
We begin with the preliminary  lemma. Here, we fix four open sets
\begin{equation}\label{eq:nested-sets-pub}
 G_0\Subset G_1\Subset G_2\Subset G_3\Subset Q_T,
\end{equation}
and obtain
\begin{lemma} 
\label{lem:bad-scale-selection-pub}
Let \(G_0\Subset G_1\Subset Q_T\). Suppose that there exist
\(\delta>0\), a sequence \(s_n\downarrow0\), functions \(u_n\in\mathcal U\) 
and pairs \(z_n^0,\widehat z_n^0\in G_0\) satisfying
\begin{equation}\label{eq:initial-bad-pair-pub}
0<d_p(z_n^0,\widehat z_n^0)\le s_n,
\qquad
|u_n(z_n^0)-u_n(\widehat z_n^0)|
\ge
\delta\,d_p(z_n^0,\widehat z_n^0)^{\alpha_c}.
\end{equation}
After passing to a subsequence, independently of
\(\varepsilon\), then  for every fixed
\(0<\varepsilon<\delta\), there exist pairs
\(z_n,\widehat z_n\in G_1\) depending on \(\varepsilon\) such that  with $r_n:=d_p(z_n,\widehat z_n),$
one has $0<r_n\le s_n$,
$Q_{nr_n}(z_n)\Subset G_1$
for all sufficiently large \(n\), and
\begin{equation}\label{eq:selected-bad-pair-pub}
|u_n(z_n)-u_n(\widehat z_n)|
\ge
\varepsilon r_n^{\alpha_c}.
\end{equation}
Moreover,
\begin{equation}\label{eq:selected-local-holder-pub}
\frac{|u_n(z)-u_n(\bar z)|}
     {d_p(z,\bar z)^{\alpha_c}}
<\varepsilon
\end{equation}
whenever $z,\bar z\in Q_{nr_n}(z_n)$ and
 $0<d_p(z,\bar z)\le\frac{r_n}{2}.$

\end{lemma}

\begin{proof}
Since \(s_n\downarrow0\), after passing to a subsequence,
we may assume that
\begin{equation}\label{eq:nsn-zero-pub}
ns_n\rightarrow0.
\end{equation}
This subsequence is chosen independently of \(\varepsilon\).

Fix \(0<\varepsilon<\delta\), and set $r_n^0:=d_p(z_n^0,\widehat z_n^0).$
By using \eqref{eq:initial-bad-pair-pub}, we have
\begin{equation}\label{eq:initial-epsilon-bad-pub}
0<r_n^0\le s_n,
\qquad
|u_n(z_n^0)-u_n(\widehat z_n^0)|
\ge
\delta(r_n^0)^{\alpha_c}
>
\varepsilon(r_n^0)^{\alpha_c}.
\end{equation}

We first ensure that all the cylinders used below remain inside \(G_1\).
Define
\[
\rho_*:=
\inf\left\{
d_p(z,\bar z):
z\in\overline{G_0},\
\bar z\in
(\mathbb T^d\times\mathbb R)\setminus G_1
\right\}.
\]
Since \(\overline{G_0}\subset G_1\), with
\(\overline{G_0}\) compact and \(G_1\) open, one has $\rho_*>0.$
In light of \eqref{eq:nsn-zero-pub},  we
may assume that
\begin{equation}\label{eq:buffer-condition-pub}
6ns_n<\rho_*
\qquad\text{for every }n.
\end{equation}
We now construct the  pairs recursively. Suppose that $(z_n^k,\widehat z_n^k,r_n^k)$ are given. Set
\begin{align}
\mathcal A_n^k
:=
\Bigl\{
(z,\bar z)\in
Q_{nr_n^k}(z_n^k)\times Q_{nr_n^k}(z_n^k):
\;&0<d_p(z,\bar z)\le\tfrac12r_n^k,
\notag\\[-1mm]
&|u_n(z)-u_n(\bar z)|
\ge
\varepsilon d_p(z,\bar z)^{\alpha_c}
\Bigr\}.
\label{eq:bad-pair-set-pub}
\end{align}
If \(\mathcal A_n^k=\varnothing\), {then we stop}. If
\(\mathcal A_n^k\ne\varnothing\), choose
\begin{equation*}
(z_n^{k+1},\widehat z_n^{k+1})\in\mathcal A_n^k,
\qquad
r_n^{k+1}
:=
d_p(z_n^{k+1},\widehat z_n^{k+1}).
\end{equation*}
By the definition of \(\mathcal A_n^k\),
\begin{equation}\label{eq:recursive-scale-pub}
0<r_n^{k+1}\le\tfrac12r_n^k,
\qquad
|u_n(z_n^{k+1})-u_n(\widehat z_n^{k+1})|
\ge
\varepsilon(r_n^{k+1})^{\alpha_c}.
\end{equation}
Consequently, for every index \(k\),
\begin{equation}\label{eq:geometric-scale-pub}
0<r_n^k\le2^{-k}r_n^0\le2^{-k}s_n.
\end{equation}
Moreover, \eqref{eq:initial-epsilon-bad-pub} and
\eqref{eq:recursive-scale-pub} imply
\begin{equation}\label{eq:all-selected-pairs-bad-pub}
|u_n(z_n^k)-u_n(\widehat z_n^k)|
\ge
\varepsilon(r_n^k)^{\alpha_c}
\qquad\text{for every reached index }k.
\end{equation}
Next, we  prove that every cylinder used in the
construction is compactly contained in \(G_1\). By the definitions of
\(Q_a\) and \(d_p\),
\begin{equation}\label{eq:cylinder-distance-pub}
z\in Q_a(\bar z)
\quad\rightarrow\quad
d_p(z,\bar z)<2a.
\end{equation}
%Indeed, if \(z=(x,t)\) and \(\bar z=(\bar x,\bar t)\), then
%\[
%\operatorname{dist}_{\mathbb T^d}(x,\bar x)<a,
%\qquad
%|t-\bar t|^{1/2}<a,
%\]
%and hence \(d_p(z,\bar z)<2a\).
Since \(z_n^{k+1}\in Q_{nr_n^k}(z_n^k)\),
\eqref{eq:cylinder-distance-pub} gives $d_p(z_n^{k+1},z_n^k)<2nr_n^k.$
It then follows from \eqref{eq:geometric-scale-pub} that
\begin{align*}
d_p(z_n^k,z_n^0)
&\le
\sum_{i=0}^{k-1}d_p(z_n^{i+1},z_n^i)<
2n\sum_{i=0}^{k-1}r_n^i
\le
2ns_n\sum_{i=0}^{k-1}2^{-i}
\le4ns_n.
\end{align*}
Therefore, if \(z\in Q_{nr_n^k}(z_n^k)\), then
\begin{align*}
d_p(z,z_n^0)
&\le
d_p(z,z_n^k)+d_p(z_n^k,z_n^0)
<
2nr_n^k+4ns_n
\le6ns_n
<\rho_*,
\end{align*}
where the last inequality follows from
\eqref{eq:buffer-condition-pub}. Since \(z_n^0\in G_0\), the definition
of \(\rho_*\) implies $Q_{nr_n^k}(z_n^k)\Subset G_1.$
%In particular, all the pairs and cylinders used in the recursive
%construction lie in \(G_1\).

We now prove that, for every fixed \(n\), the construction stops
after finitely many steps. Suppose otherwise,  then
\eqref{eq:geometric-scale-pub} gives $r_n^k\rightarrow0$
$\text{ as }k\to\infty,$
whereas \eqref{eq:all-selected-pairs-bad-pub} gives
\[
\frac{
|u_n(z_n^k)-u_n(\widehat z_n^k)|
}{
d_p(z_n^k,\widehat z_n^k)^{\alpha_c}
}
\ge\varepsilon
\qquad\text{for every }k.
\]
All the  pairs  above belong to \(G_1\), so this contradicts
Lemma~\ref{lem:individual-little-pub}, applied to the fixed function
\(u_n\), which implies
\[
\lim_{\rho\downarrow0}
\sup_{\substack{z,\bar z\in G_1\\
0<d_p(z,\bar z)\le\rho}}
\frac{|u_n(z)-u_n(\bar z)|}
     {d_p(z,\bar z)^{\alpha_c}}
=0.
\]
Thus,  the construction stops  at some finite index \(K_n\).

Set
\begin{equation*}
z_n:=z_n^{K_n},
\qquad
\widehat z_n:=\widehat z_n^{K_n},
\qquad
r_n:=r_n^{K_n}.
\end{equation*}
In light of \eqref{eq:geometric-scale-pub}, we find $0<r_n\le s_n,$
and \eqref{eq:all-selected-pairs-bad-pub}  proves \eqref{eq:selected-bad-pair-pub}.  Moreover, since the construction stops at \(K_n\), one has $\mathcal A_n^{K_n}=\varnothing.$
Thanks to \eqref{eq:bad-pair-set-pub}, we find \eqref{eq:selected-local-holder-pub} holds. Finally, since  $Q_{nr_n}(z_n)\Subset G_1$, the proof is complete.
\end{proof}

We next give the first blow-up sequence.  Letting \(z_n=(x_n,t_n)\) and \(f_n:=f_{u_n}\) given in Lemma \ref{lem:bad-scale-selection-pub}, we define
\begin{equation}\label{eq:first-rescaling-pub}
 v_n(y,s):=
 \frac{u_n(x_n+r_ny,t_n+r_n^2s)-u_n(x_n,t_n)}
      {\varepsilon r_n^{\ac}},
 ~
 g_n(y,s):=
 \varepsilon^{-1}r_n^{\gp}
 f_n(x_n+r_ny,t_n+r_n^2s),
\end{equation}
with the domain $ \Omega_n:=
 \{(y,s):(x_n+r_ny,t_n+r_n^2s)\in G_2\}.$
Since \(z_n\in G_1\Subset G_2\) and \(r_n\to0\), by using (\ref{eq:main-equation-pub}), one has      \(v_n\) satisfies
\begin{equation}\label{eq:first-rescaled-equation-pub}
 (v_n)_s-\Delta v_n+\lambda|Dv_n|^\gamma=g_n
 \quad\text{in }\Omega_n,
 \qquad \lambda:=\varepsilon^{\gamma-1}.
\end{equation}
Let \(\zeta_n\) be the image of \(\widehat z_n\) under this scaling.  Then
\begin{equation}\label{eq:first-normalization-pub}
 v_n(0,0)=0,
 \qquad
 |v_n(\zeta_n)|\ge1,
 \qquad
 \dpd(\zeta_n,(0,0))=1.
\end{equation}
Moreover, \eqref{eq:selected-local-holder-pub} becomes
\begin{equation}\label{eq:first-local-holder-pub}
 |v_n(z)-v_n(\bar z)|
 \le\dpd(z,\bar z)^{\ac}
\end{equation}
whenever \(z,\bar z\in Q_n\) and
\(\dpd(z,\bar z)\le1/2\).  Finally,
Theorem~\ref{thm:critical-holder-pub}, applied on \(G_2\Subset G_3\),
gives
\begin{equation}\label{eq:first-global-holder-pub}
 [v_n]_{C^{\ac,\ac/2}(\Omega_n)}
 \le H_{G_2}/\varepsilon.
\end{equation}
We next prove the following vanishing lemma for the rescaled source term
in \eqref{eq:first-rescaled-equation-pub}.
\begin{lemma}
\label{lem:data-disappearance-pub}
Assume all the hypotheses of Theorem \ref{thm:main-pub} hold. For every fixed \(L>0\), we have as $n\rightarrow \infty,$
\begin{equation}\label{eq:first-data-vanish-pub}
 g_n\rightarrow0
\text{ in }L^{\qc}(Q_L).
\end{equation}
More precisely,
\begin{equation*}
 \sup_{\substack{\xi=(\xi_x,\xi_t),\ 0<\rho\le1\\
 Q_L\ni z\mapsto(\xi_x+\rho z_x,\xi_t+\rho^2z_t)\in\Omega_n}}
 \left\|
 \rho^{\gp}g_n(\xi_x+\rho\,\cdot,\xi_t+\rho^2\,\cdot)
 \right\|_{L^{\qc}(Q_L)}
 \rightarrow0.
\end{equation*}
\end{lemma}

\begin{proof}
By using \(\gp\qc=d+2\), we obtain from (\ref{vanishing-omega}) that
\begin{align*}
 \|g_n\|_{L^{\qc}(Q_L)}^{\qc}
 &=\varepsilon^{-\qc}
 \int_{E_{n,L}}|f_n|^{\qc}\dd x\dd t 
 \le\varepsilon^{-\qc}
 \omega_{\F}(C_Lr_n^{d+2})\rightarrow0,
\end{align*}
where \(|E_{n,L}|\le C_Lr_n^{d+2}\).  After the rescaling again, we have  
\begin{align*}
 &\left\|
 \rho^{\gp}g_n(\xi_x+\rho\,\cdot,\xi_t+\rho^2\,\cdot)
 \right\|_{L^{\qc}(Q_L)}^{\qc}
 \notag\\
 &\qquad\le
 \varepsilon^{-\qc}\omega_{\F}
 \bigl(C_L(r_n\rho)^{d+2}\bigr)
 \le\varepsilon^{-\qc}\omega_{\F}(C_Lr_n^{d+2})
 \rightarrow0.
\end{align*}
This finishes the proof of this lemma.
%The bound is independent of the admissible centre \(\xi\) and of
%\(0<\rho\le1\), proving both assertions.
\end{proof}

In the next subsection, we establish
a uniform local energy bound for the blow-up sequence, which provides the
weak compactness needed in the subsequent argument.
\subsection{Local energy estimate}

For a measurable set \(A\) contained in the domain of \(v\), define the
 energy
\begin{equation}\label{eq:critical-energy-pub}
 \Energy(v;A):=
 \int_A\left(
 |v_t|^{\qc}+|D^2v|^{\qc}+|Dv|^{\gamma\qc}
 \right)\dd x\dd t.
\end{equation}
We note that  if $w(y,s)=\rho^{-\ac}
 \bigl(v(\xi_x+\rho y,\xi_t+\rho^2s)-v(\xi)\bigr),$
then $2-\ac=\gp$ and $\gamma(1-\ac)=\gp$ imply $ \Energy(w;Q_L)=\Energy(v;Q_{L\rho}(\xi)).$  The energy estimate of the blow-up sequence is summarized in the following proposition.
%Indeed, \(w_s\) and \(D^2w\) acquire %the factor
%\(\rho^{2-\ac}=\rho^{\gp}\), whereas \(Dw\) acquires the factor
%\(\rho^{1-\ac}\).  After taking the powers in
%\eqref{eq:critical-energy-pub}, each term has factor \(\rho^{d+2}\),
%which cancels the parabolic Jacobian.

%The first-scale sequence in Section~\ref{sec:first-scale-pub} is
%admissible by \eqref{eq:first-local-holder-pub} and
%Lemma~\ref{lem:data-disappearance-pub}.

\begin{proposition}
\label{prop:universal-energy-pub}
Let
\((v_n,g_n,\lambda_n,\Omega_n)\) satisfy
\begin{align}\label{vneqprop33}
(v_n)_t-\Delta v_n+\lambda_n|Dv_n|^\gamma=g_n
\quad\text{in }\Omega_n,
\qquad
0\le\lambda_n\le1,
\end{align}
where   
\(v_n\in W^{2,1}_{q_c,\mathrm{loc}}(\Omega_n)\). {Assume that
every compact
subset of \(\mathbb R^d\times\mathbb R\) is contained in \(\Omega_n\) for all
sufficiently large \(n\).} Assume that, for every
fixed \(L>0\), (\ref{eq:first-local-holder-pub}) holds for all sufficiently large \(n\), whenever $z,\bar z\in Q_L$,
 $d_p(z,\bar z)\le\frac12$
and
\begin{equation}\label{eq:admissible-data-pub}
\sup_{\substack{\xi=(\xi_x,\xi_t),\ 0<\rho\le1\\
(\xi_x+\rho y,\xi_t+\rho^2s)\in\Omega_n\
\text{for every }(y,s)\in Q_L}}
\left\|
\rho^{\gamma'}
g_n(\xi_x+\rho y,\xi_t+\rho^2s)
\right\|_{L^{q_c}_{(y,s)}(Q_L)}
\rightarrow0.
\end{equation}
Then there exists \(M_0=M_0(d,\gamma)>0\) such that  for every fixed
\(R>0\),
\begin{equation}\label{eq:universal-energy-pub}
\sup_{\zeta\in Q_R}
\mathcal E\bigl(v_n;Q_2(\zeta)\bigr)
\le M_0
\end{equation}
for all sufficiently large \(n\).
{The constant \(M_0\) is independent of
\(R\) and the coefficient sequence \((\lambda_n)\subset[0,1]\).}
\end{proposition}

\begin{proof}
We first exclude unbounded energy and then
prove that the tail bound is uniform.

\textbf{Step 1: boundedness}   Suppose that, for some fixed \(R\), after extraction,
\begin{equation}\label{eq:energy-divergence-pub}
 \Energy(v_n;Q_R)\rightarrow\infty.
\end{equation}
Fix a number \(0<\eta_0<1\), which  will be chosen below.  Choose
\(a_j\downarrow0\) so that \(ja_j\to0\).  For \(a>0\), define the covering number of
\(\overline{Q_R}\) by
\[
N(Q_R,a)
:=
\min\left\{
N\in\mathbb N:
\begin{array}{l}
\text{there exist }\xi_1,\ldots,\xi_N\in\overline{Q_R}
\text{ such that }
\overline{Q_R}\subset
\bigcup_{\ell=1}^N Q_a(\xi_\ell)
\end{array}
\right\}.
\]
Set $N_j:=N(Q_R,a_j),$ then
\[
N_j\le
C_d\left(1+\frac{R}{a_j}\right)^{d+2}.
\]
Choose centres
\(\xi_{j,1},\ldots,\xi_{j,N_j}\in\overline{Q_R}\) such that  $\overline{Q_R}
\subset
\bigcup_{\ell=1}^{N_j}Q_{a_j}(\xi_{j,\ell}).$
Using \eqref{eq:energy-divergence-pub}, we choose \(n(j)\) sufficiently large that $\mathcal E(v_{n(j)};Q_R)>jN_j.$  We also choose $n(j)$ so large that
$Q_{j+2R+4}\Subset\Omega_{n(j)}$,
the estimate \eqref{eq:first-local-holder-pub} holds on
$Q_{j+2R+4}$, and the quantity in
\eqref{eq:admissible-data-pub} is below $1/j$ for every integer
$1\le L\le j+2R+4$.
Hence,
\[
\mathcal E(v_{n(j)};Q_R)
\le
\sum_{\ell=1}^{N_j}
\mathcal E\bigl(
v_{n(j)};Q_{a_j}(\xi_{j,\ell})
\bigr).
\]
Consequently, there exists
\(\ell_j\in\{1,\ldots,N_j\}\) such that $\mathcal E\bigl(
v_{n(j)};Q_{a_j}(\xi_{j,\ell_j})
\bigr)>j,$
where $\xi_j^0:=\xi_{j,\ell_j}.$    

Starting from $(\xi_j^0,\rho_j^0):=(\xi_j^0,a_j),$
we construct recursively a sequence
\((\xi_j^k,\rho_j^k)\). By the choice of \(\xi_j^0\),
\[
\mathcal E\bigl(v_{n(j)};Q_{\rho_j^0}(\xi_j^0)\bigr)
=
\mathcal E\bigl(v_{n(j)};Q_{a_j}(\xi_j^0)\bigr)
>j>\eta_0
\]
for all sufficiently large \(j\).  Suppose that \((\xi_j^k,\rho_j^k)\) is given. Introduce the set
\begin{equation*}
\mathcal C_j^k
:=
\left\{
\xi:
d_p(\xi,\xi_j^k)\le j\rho_j^k,\quad
\mathcal E\bigl(
v_{n(j)};Q_{\rho_j^k/2}(\xi)
\bigr)>\eta_0
\right\}.
\end{equation*}
If \(\mathcal C_j^k=\varnothing\), the construction stops. Otherwise,
choose $\xi_j^{k+1}\in\mathcal C_j^k$
and set $\rho_j^{k+1}:=\frac{\rho_j^k}{2}.$
By our construction,
\begin{equation}\label{eq:selected-energy-lower-bound-pub}
\mathcal E\bigl(
v_{n(j)};Q_{\rho_j^{k+1}}(\xi_j^{k+1})
\bigr)>\eta_0.
\end{equation}
Moreover, $\rho_j^k=2^{-k}a_j$
for every index.

We first show that all the cylinders used in the construction
are contained in a fixed compact subset of \(\Omega_{n(j)}\). For every
\(k\ge1\),
\begin{align}
d_p(\xi_j^k,\xi_j^0)
&\le
\sum_{i=0}^{k-1}
d_p(\xi_j^{i+1},\xi_j^i)
\le
j\sum_{i=0}^{k-1}\rho_j^i
=
ja_j\sum_{i=0}^{k-1}2^{-i}
\le2ja_j.
\label{eq:energy-centre-displacement-pub}
\end{align}
If \(z\in Q_{\rho_j^k}(\xi_j^k)\), then
\(d_p(z,\xi_j^k)<2\rho_j^k\le2a_j\), and hence
\begin{equation*}
d_p(z,\xi_j^0)
\le
d_p(z,\xi_j^k)+d_p(\xi_j^k,\xi_j^0)
<
2a_j+2ja_j.
\end{equation*}
Since \(ja_j\to0\), for all sufficiently large \(j\), $(2j+2)a_j<1.$
Since \(\xi_j^0\in\overline{Q_R}\), it follows that every cylinder
in the construction is contained in \(Q_{R+1}\). We chose
\(n(j)\) sufficiently large such that $\overline{Q_{R+1}}\Subset\Omega_{n(j)}.$ It shows that all cylinders  are contained in a  compact subset of
\(\Omega_{n(j)}\).

We now prove that the construction stops after finitely many
steps. Fix \(j\) and suppose, by contradiction, that it does not
stop. Then $\rho_j^k=2^{-k}a_j\rightarrow0.$
Furthermore, for \(m>k\),
\begin{align*}
d_p(\xi_j^m,\xi_j^k)
&\le
\sum_{i=k}^{m-1}
d_p(\xi_j^{i+1},\xi_j^i)\le
j\sum_{i=k}^{m-1}\rho_j^i
\le
2j\rho_j^k
\rightarrow0
\qquad\text{as }k\to\infty.
\end{align*}
Hence \((\xi_j^k)_k\) is a Cauchy sequence and therefore there is a  limiting point $\xi_j^\infty\in\overline{Q_{R+1}}
\Subset\Omega_{n(j)}.$  For this fixed \(j\), define the energy 
\[
e_j
:=
|(v_{n(j)})_t|^{q_c}
+|D^2v_{n(j)}|^{q_c}
+|Dv_{n(j)}|^{\gamma q_c}.
\]
Since \(v_{n(j)}\in
W^{2,1}_{q_c,\mathrm{loc}}(\Omega_{n(j)})\) and $W^{2,1}_{q_c,\mathrm{loc}}
\hookrightarrow
W^{1,0}_{\gamma q_c,\mathrm{loc}},$
one has $e_j\in L^1(Q_{R+1}).$
The Lebesgue integral of \(e_j\) is therefore absolutely continuous.
On the other hand, $\bigl|Q_{\rho_j^k}(\xi_j^k)\bigr|
=
|Q_1|(\rho_j^k)^{d+2}
\rightarrow0.$
Taking \(\varepsilon=\eta_0\), we obtain, for all sufficiently large
\(k\),
\[
\mathcal E\bigl(
v_{n(j)};Q_{\rho_j^k}(\xi_j^k)
\bigr)
=
\int_{Q_{\rho_j^k}(\xi_j^k)}e_j
<\eta_0,
\]
which contradicts
\eqref{eq:selected-energy-lower-bound-pub}. Hence the recursive
construction must stop after finitely many steps.

Let \(K_j\) denote the terminal index and set
\[
\xi_j:=\xi_j^{K_j},
\qquad
\rho_j:=\rho_j^{K_j}.
\]
Then
\begin{equation}\label{eq:energy-lower-normalization-pub}
\mathcal E\bigl(v_{n(j)};Q_{\rho_j}(\xi_j)\bigr)>\eta_0,
\end{equation}
whereas the stopping condition
\(\mathcal C_j^{K_j}=\varnothing\) gives
\begin{equation}\label{eq:energy-local-upper-pub}
\sup_{d_p(\xi,\xi_j)\le j\rho_j}
\mathcal E\bigl(
v_{n(j)};Q_{\rho_j/2}(\xi)
\bigr)
\le\eta_0.
\end{equation}
Moreover, \(\rho_j\le a_j\to0\), and
\eqref{eq:energy-centre-displacement-pub} keeps \(\xi_j\) in a fixed
compact set.

Define the second blow-up
\begin{align*}
 &w_j(y,s):=\rho_j^{-\ac}
 \bigl(v_{n(j)}(\xi_{j,x}+\rho_jy,
                 \xi_{j,t}+\rho_j^2s)-v_{n(j)}(\xi_j)\bigr),\nonumber\\~ &k_j(y,s):=\rho_j^{\gp}
 g_{n(j)}(\xi_{j,x}+\rho_jy,
          \xi_{j,t}+\rho_j^2s).
\end{align*}
Then, we have from (\ref{vneqprop33}) that
\begin{equation}\label{eq:second-equation-pub}
 (w_j)_s-\Delta w_j+\lambda_{n(j)}|Dw_j|^\gamma=k_j.
\end{equation}
Then, thanks to \eqref{eq:admissible-data-pub},  one has
\begin{equation}\label{eq:second-data-zero-pub}
 k_j\rightarrow0
 \quad\text{in }L^{\qc}_{\mathrm{loc}}(\R^d\times\R).
\end{equation}
Invoking \eqref{eq:energy-lower-normalization-pub}, we have
\begin{equation}\label{eq:second-energy-lower-pub}
 \Energy(w_j;Q_1)>\eta_0.
\end{equation}
Furthermore, for every fixed \(L\) and all sufficiently large \(j\),
\eqref{eq:energy-local-upper-pub} yields
\begin{equation}\label{eq:second-energy-local-pub}
 \sup_{z\in Q_L}\Energy(w_j;Q_{1/2}(z))\le\eta_0.
\end{equation}

Fix \(L>0\). For all sufficiently large \(j\), choose a covering
\[
Q_L\subset
\bigcup_{\ell=1}^{N_L}Q_{1/2}(\zeta_\ell),
\qquad
N_L\le C_d(1+L)^{d+2},
\]
whose centres satisfy \(d_p(\zeta_\ell,0)\le j\). It follows from
\eqref{eq:second-energy-local-pub} that
\begin{align}
\mathcal E(w_j;Q_L)
&\le
\sum_{\ell=1}^{N_L}
\mathcal E\bigl(w_j;Q_{1/2}(\zeta_\ell)\bigr)
\le N_L\eta_0
\le C_L\eta_0.
\label{eq:second-energy-compact-pub}
\end{align}
Similarly, every cylinder \(Q_2(z)\) can be covered by at most \(N_d\)
cylinders of radius \(1/2\), where \(N_d\) depends only on \(d\).
For \(z\in Q_L\) and all sufficiently large \(j\), the centres of these
covering cylinders satisfy \(d_p(\zeta,0)\le j\). Hence
\begin{equation}\label{eq:second-energy-radius-two-pub}
\sup_{z\in Q_L}
\mathcal E\bigl(w_j;Q_2(z)\bigr)
\le N_d\eta_0
\le C_d\eta_0,
\end{equation}
where \(C_d\) is independent of \(L\).

Since \(\rho_j\to0\),   H\"older estimate
\eqref{eq:first-local-holder-pub} implies, for every fixed \(L\), $
 [w_j]_{C^{\ac,\ac/2}(Q_L)}\le1$
for all large \(j\).  Since \(w_j(0,0)=0\), Arzel\`a--Ascoli and a
diagonal extraction give local uniform convergence to a function \(w\)
with a finite \(\ac\)-H\"older semi-norm.  We also
extract so that \(\lambda_{n(j)}\to\lambda_\infty\in[0,1]\).

We next prove strong  compactness of ${(w_j)}$.  For indices \(j,\ell\), set
\(Z_{j\ell}=w_j-w_\ell\) and
\begin{equation*}
 A_{j\ell}:=\gamma\int_0^1
 |Dw_\ell+\theta(Dw_j-Dw_\ell)|^{\gamma-2}
 \bigl(Dw_\ell+\theta(Dw_j-Dw_\ell)\bigr)\dd\theta.
\end{equation*}
Subtracting the equations gives
\begin{equation*}
 (Z_{j\ell})_t-\Delta Z_{j\ell}
 +\mathbf{B}_{j\ell}\cdot DZ_{j\ell}=R_{j\ell},
\end{equation*}
where
\begin{equation*}
 \mathbf{B}_{j\ell}=\lambda_{n(j)}A_{j\ell},
 \qquad
 R_{j\ell}=k_j-k_\ell
 +(\lambda_{n(\ell)}-\lambda_{n(j)})|Dw_\ell|^\gamma.
\end{equation*}
The identity
\(
 \gamma\qc=(d+2)(\gamma-1)
\)
and \eqref{eq:second-energy-radius-two-pub} imply
\begin{equation*}
 \|\mathbf{B}_{j\ell}\|_{L^{d+2}(Q_2(z))}
 \le C_d\eta_0^{1/(d+2)}
\end{equation*}
for every fixed compact set and all large \(j,\ell\).  Choose
\(\eta_0=\eta_0(d,\gamma)\) so small that the right-hand side is below
\(\kappa_{\qc}\) in Lemma~\ref{lem:small-drift-pub}.  Moreover,
\eqref{eq:second-data-zero-pub}, and
\eqref{eq:second-energy-compact-pub} give
\begin{equation*}
 R_{j\ell}\rightarrow0
 \quad\text{in }L^{\qc}_{\mathrm{loc}}
 \text{ as }j,\ell\to\infty.
\end{equation*}
Indeed, on \(Q_L\),
\[
 \|(\lambda_{n(\ell)}-\lambda_{n(j)})|Dw_\ell|^\gamma\|_{L^{\qc}}
 \le |\lambda_{n(\ell)}-\lambda_{n(j)}|
      (C_L\eta_0)^{1/\qc}.
\]
Local uniform convergence gives
\(Z_{j\ell}\to0\) in \(L^{\qc}_{\mathrm{loc}}\).  Applying
Lemma~\ref{lem:small-drift-pub}(i) on a finite cover therefore yields
\begin{equation}\label{eq:second-strong-W-pub}
 w_j\rightarrow w
 \quad\text{ in }W^{2,1}_{\qc,\mathrm{loc}}.
\end{equation}
The continuous Sobolev embedding gives
\begin{equation}\label{eq:second-strong-gradient-pub}
 Dw_j\rightarrow Dw
 \quad\text{ in }L^{\gamma\qc}_{\mathrm{loc}}.
\end{equation}
We can now pass to the limit in \eqref{eq:second-equation-pub} and obtain
\[
 w_t-\Delta w+\lambda_\infty|Dw|^\gamma=0
 \quad\text{in }\R^d\times\R.
\]
Lemma~\ref{lem:endpoint-liouville-pub} implies  \(w\) is a constant.  On the
other hand, \eqref{eq:second-strong-W-pub},
\eqref{eq:second-strong-gradient-pub}, and
\eqref{eq:second-energy-lower-pub} imply
\(\Energy(w;Q_1)\ge\eta_0\), which is impossible.  Hence, the  energy is bounded on a fixed compact set.

\textbf{Step 2: uniformity of the bound.}
We now prove that the bound obtained in Step~1 can be chosen independently
of the   sequence. Suppose otherwise. Then, for every
$j\in\mathbb N$, there exist sequences
\[
 \bigl(v_n^j,g_n^j,\lambda_n^j,\Omega_n^j\bigr)_{n\in\mathbb N}
\]
satisfying all the hypotheses of the proposition, a radius $R_j>0$, and
infinitely many indices $n$ such that
\begin{equation}\label{eq:failure-uniform-tail-pub}
 \sup_{\zeta\in Q_{R_j}}
 \Energy\bigl(v_n^j;Q_2(\zeta)\bigr)>j.
\end{equation}
In particular,
\[
 (v_n^j)_t-\Delta v_n^j
 +\lambda_n^j|Dv_n^j|^\gamma=g_n^j
 \quad\text{in }\Omega_n^j,
 \qquad 0\leq\lambda_n^j\leq1.
\]
Moreover, the domains $\Omega_n^j$ exhaust
$\mathbb R^d\times\mathbb R$ in the sense that every compact subset is
contained in $\Omega^j_n$ for all sufficiently large $n$. For every fixed
$L>0$, the normalized local H\"older estimate holds in $Q_L$ for all
sufficiently large $n$ and the scale-uniform vanishing condition
\eqref{eq:admissible-data-pub} is satisfied.

For each $j$, choose one of these indices, denoted by $n(j)$, sufficiently
large that $Q_{R_j+j+3}\Subset\Omega_{n(j)}^j$
and $|v_{n(j)}^j(z)-v_{n(j)}^j(\bar z)|
 \le d_p(z,\bar z)^{\alpha_c}$
whenever $z,\bar z\in Q_{R_j+j+3}$ and
$d_p(z,\bar z)\le1/2$. We also require that, for every integer
$1\le L\le j$, the supremum on the left-hand side of
\eqref{eq:admissible-data-pub}, with
$(g_n,\Omega_n)=(g_{n(j)}^j,\Omega_{n(j)}^j)$, is at most $1/j$.
Finally, by \eqref{eq:failure-uniform-tail-pub}, we may choose
$\zeta_j\in Q_{R_j}$ such that $\Energy\bigl(v_{n(j)}^j;Q_2(\zeta_j)\bigr)>j.$  Define $\widetilde\Omega_j:=\Omega_{n(j)}^j-\zeta_j$
and
\[
 \widetilde v_j(z)
 :=v_{n(j)}^j(z+\zeta_j)-v_{n(j)}^j(\zeta_j),
 \qquad
 \widetilde g_j(z):=g_{n(j)}^j(z+\zeta_j),
 \qquad
 \widetilde\lambda_j:=\lambda_{n(j)}^j.
\]
Translation invariance gives
\begin{equation*}
 (\widetilde v_j)_t-\Delta\widetilde v_j
 +\widetilde\lambda_j|D\widetilde v_j|^\gamma
 =\widetilde g_j
 \quad\text{in }\widetilde\Omega_j,
 \qquad
 0\leq\widetilde\lambda_j\leq1.
\end{equation*}
Since $\zeta_j\in Q_{R_j}$ and
$Q_{R_j+j+3}\Subset\Omega_{n(j)}^j$, we have $Q_{j+2}\Subset\widetilde\Omega_j.$  Consequently,    every compact subset of ${\mathbb R^d\times\mathbb R}$ is
contained in $\tilde \Omega_j$ for all sufficiently large $j$.  Let $L>0$ be fixed. For every sufficiently large $j$, we have $L\leq j$.
The facts $|v_{n(j)}^j(z)-v_{n(j)}^j(\bar z)|
 \le d_p(z,\bar z)^{\alpha_c}$
whenever $z,\bar z\in Q_{R_j+j+3}$ with
$d_p(z,\bar z)\le1/2$ and \eqref{eq:admissible-data-pub}, together with translation
invariance, then give $|\widetilde v_j(z)-\widetilde v_j(\bar z)|
 \leq d_p(z,\bar z)^{\alpha_c}$
whenever $z,\bar z\in Q_L$ and $d_p(z,\bar z)\leq1/2$, and
\[
\sup_{\substack{\xi=(\xi_x,\xi_t),\ 0<\rho\leq1\\
(\xi_x+\rho y,\xi_t+\rho^2s)\in\widetilde\Omega_j\\
\text{for every }(y,s)\in Q_L}}
\left\|
 \rho^{\gamma'}
 \widetilde g_j(\xi_x+\rho y,\xi_t+\rho^2s)
\right\|_{L^{q_c}_{(y,s)}(Q_L)}
\leq\frac1j.
\]
Thus, the sequence
$(\widetilde v_j,\widetilde g_j,\widetilde\lambda_j,
\widetilde\Omega_j)$ satisfies explicitly all the hypotheses used in
Step~1.

Finally, translation invariance of the critical energy and property
$\Energy\bigl(v_{n(j)}^j;Q_2(\zeta_j)\bigr)>j$ imply
\[
 \Energy(\widetilde v_j;Q_2)
 =
 \Energy\bigl(v_{n(j)}^j;Q_2(\zeta_j)\bigr)>j.
\]
This contradicts Step~1, applied to the single sequence
$(\widetilde v_j)_j$ on the fixed cylinder $Q_2$. Therefore, the local
critical-energy bound can be chosen uniformly, and there exists
$M_0=M_0(d,\gamma)$ such that \eqref{eq:universal-energy-pub} holds.
\end{proof}

\begin{remark}
For the sequence arising from the first blow-up, the coefficient
of the Hamiltonian is
\(\lambda=\varepsilon^{\gamma-1}\), which can be made arbitrarily small by
choosing \(\varepsilon\) sufficiently small.  The local parabolic estimate
and the critical interpolation inequality then allow the nonlinear term to
be absorbed, yielding a local critical-energy bound for this 
sequence.  This shorter argument is sufficient for the first blow-up, but it
depends essentially on the smallness of \(\lambda\).

Proposition~\ref{prop:universal-energy-pub} establishes a stronger statement:
the same type of local energy bound remains valid for arbitrary coefficient
sequences \((\lambda_n)\subset[0,1]\), with a constant independent of the
coefficients.  In its proof, energy normalization
at the second blow-up scale provides the smallness required for
the linearized drift.

This coefficient-uniform estimate is also used in the existence
proof of Theorem~\ref{thm:critical-existence-pub}. Applying it with
$\lambda_n=1$ gives local energy bounds for the rescaled approximating
solutions. Together with the uniform little-H\"older estimate, these bounds
rule out concentration of the critical energy and yield the strong
compactness needed to pass to the limit in the regularized equations
obtained by mollifying the source term and the initial datum.
\end{remark}

With the aid of Proposition \ref{prop:universal-energy-pub}, we next show the strong convergence of the first blow-up sequence  by using Lemma \ref{lem:small-drift-pub}.

\begin{proposition}
\label{prop:first-compactness-pub}

Let \((v_n,g_n,\lambda,\Omega_n)\) be the sequence given in Proposition \ref{prop:universal-energy-pub}
with a fixed coefficient \(0<\lambda\le\lambda_*\), where constant \(\lambda_*=\lambda_*(d,\gamma)>0\).  Assume
\(v_n(0,0)=0\).  Then, after passing to a subsequence, there is a
\(v\in W^{2,1}_{\qc,\mathrm{loc}}(\R^d\times\R)\) such that
\begin{equation}\label{eq:first-strong-W-pub}
 v_n\rightarrow v
 \quad\text{ in }W^{2,1}_{\qc,\mathrm{loc}},
\end{equation}
and
\begin{equation}\label{eq:first-strong-gradient-pub}
 Dv_n\rightarrow Dv
 \quad\text{ in }L^{\gamma\qc}_{\mathrm{loc}}.
\end{equation}
\end{proposition}

\begin{proof}

The local H\"older estimate (\ref{eq:first-local-holder-pub}) and the normalization
\(v_n(0,0)=0\) first give local uniform compactness. More precisely,
fix \(L>0\). For all sufficiently large \(n\), any point \(z\in Q_L\)
can be connected to \((0,0)\) by a finite chain $(0,0)=z_0,z_1,\ldots,z_{N_L}=z$
contained in \(Q_{L+1}\), where $d_p(z_i,z_{i-1})\le\frac12,$
$N_L\le C_L.$ Consequently,
\[
|v_n(z)|
\le
\sum_{i=1}^{N_L}
|v_n(z_i)-v_n(z_{i-1})|
\le
\sum_{i=1}^{N_L}
d_p(z_i,z_{i-1})^{\alpha_c}
\le C_L.
\]
The H\"older estimate again gives equicontinuity on \(Q_L\).
Therefore, by the Arzel\`a--Ascoli theorem and a diagonal argument,
there exists a continuous function \(v\) such that $v_n\rightarrow v$
{locally uniformly in }$\mathbb R^d\times\mathbb R$.

We next prove strong gradient compactness. For \(n,m\in\mathbb N\),
set $Z_{nm}:=v_n-v_m.$  Define 
\[
A_{nm}
:=
\gamma\int_0^1
|Dv_m+\theta(Dv_n-Dv_m)|^{\gamma-2}
\bigl(Dv_m+\theta(Dv_n-Dv_m)\bigr)
\,d\theta.
\]
Subtracting the equations satisfied by \(v_n\) and \(v_m\), we obtain
\begin{equation}\label{eq:first-difference-pub}
(Z_{nm})_t-\Delta Z_{nm}
+\mathbf B_{nm}\cdot DZ_{nm}
=
g_n-g_m,
\qquad
\mathbf B_{nm}:=\lambda A_{nm}.
\end{equation}
 For fixed \(L>0\). Since the domains \(\Omega_n\) exhaust
\(\mathbb R^d\times\mathbb R\), for all sufficiently large \(n,m\), $Q_2(\zeta)\Subset\Omega_n\cap\Omega_m$
\text{for every }$\zeta\in Q_L.$
Proposition~\ref{prop:universal-energy-pub} gives
\[
\sup_{\zeta\in Q_L}
\int_{Q_2(\zeta)}
\left(
|Dv_n|^{\gamma q_c}
+
|Dv_m|^{\gamma q_c}
\right)
\le2M_0.
\]
Using $\gamma q_c=(d+2)(\gamma-1),$
we obtain, for every \(\zeta\in Q_L\),
\begin{align*}
\|\mathbf B_{nm}\|_{L^{d+2}(Q_2(\zeta))}
&\le
C_\gamma\lambda
\left(
\|Dv_n\|_{L^{\gamma q_c}(Q_2(\zeta))}^{\gamma-1}
+
\|Dv_m\|_{L^{\gamma q_c}(Q_2(\zeta))}^{\gamma-1}
\right)\\
&=
C_\gamma\lambda
\left[
\left(
\int_{Q_2(\zeta)}|Dv_n|^{\gamma q_c}
\right)^{1/(d+2)}
+
\left(
\int_{Q_2(\zeta)}|Dv_m|^{\gamma q_c}
\right)^{1/(d+2)}
\right]\\
&\le
C_\gamma\lambda M_0^{1/(d+2)}.
\end{align*}
Thus, for every fixed \(L>0\),
\begin{equation*}
\sup_{\zeta\in Q_L}
\|\mathbf B_{nm}\|_{L^{d+2}(Q_2(\zeta))}
\le
C_\gamma\lambda M_0^{1/(d+2)}
\end{equation*}
for all sufficiently large \(n,m\).  Choose $\lambda_*:=
 \min\left\{1,
 \frac{\kappa_{\qc}}{2C_{\gamma}M_0^{1/(d+2)}}\right\},$
then the drift in \eqref{eq:first-difference-pub} satisfies the smallness
assumption in conclusion (i) of  Lemma~\ref{lem:small-drift-pub}.  By using
\eqref{eq:admissible-data-pub}, \(g_n-g_m\to0\) in
\(L^{\qc}_{\mathrm{loc}}\), while local uniform convergence gives
\(Z_{nm}\to0\) in \(L^{\qc}_{\mathrm{loc}}\).  Applying conclusion {(i)} in Lemma \ref{lem:small-drift-pub} on a finite covering implies \((v_n)\) is a Cauchy sequence in
\(W^{2,1}_{\qc}\) on every compact set.  This proves
\eqref{eq:first-strong-W-pub}.  The continuous Sobolev
embedding $W^{2,1}_{\qc}\hookrightarrow W^{1,0}_{\gamma\qc}$
then gives \eqref{eq:first-strong-gradient-pub}.
\end{proof}

%With Proposition \ref{prop:first-compactness-pub}, we have the following corollary concerning the equi
%\begin{corollary} 
%\label{cor:no-defect-pub}
%For every first bad-scale sequence with
%\(\varepsilon^{\gamma-1}\le\lambda_*\), and every fixed \(R>0\),
%the family
%\begin{equation}\label{eq:no-defect-pub}
% \{|Dv_n|^{\gamma\qc}:n\in\mathbb N\}
%\end{equation}
%is uniformly equi-integrable on \%(Q_R\).
%\end{corollary}

%\begin{proof}
%Proposition~\ref{prop:first-compactness-pub} shows that every subsequence
%has a further subsequence converging strongly in
%\(L^{\gamma\qc}(Q_R)\).  Hence \(\{Dv_n\}\) is relatively compact in
%that space.  To see explicitly that this implies
%\eqref{eq:no-defect-pub}, cover the closure of
%\(\{Dv_n:n\in\mathbb N\}\) by finitely many
%\(L^{\gamma\qc}(Q_R)\)-balls of radius \(\delta\), with centres
%\(G_1,\ldots,G_N\).  The finitely many functions
%\(|G_i|^{\gamma\qc}\) have absolutely continuous integrals, while, if
%\(\|Dv_n-G_i\|_{L^{\gamma\qc}(Q_R)}<\delta\),
%\[
% \int_E |Dv_n|^{\gamma\qc}
 %\le2^{\gamma\qc-1}\int_E|G_i|^{\gamma\qc}
 %+2^{\gamma\qc-1}\delta^{\gamma\qc}.
%\]
%First choose \(\delta\) and then \(|E|\), uniformly in \(n\).
%\end{proof}

We are ready to prove the uniformly little-H\"older estimate of the solution $u$ to (\ref{eq:main-equation-pub}). 

\begin{theorem}
\label{thm:uniform-little-pub}
Assume all conditions of Theorem \ref{thm:main-pub} hold.  For every \(G\Subset Q_T\), we have
\begin{equation}\label{eq:uniform-little-pub}
 \lim_{r\downarrow0}
 \sup_{u\in\U}
 \sup_{\substack{z,\bar z\in G\\0<\dpd(z,\bar z)\le r}}
 \frac{|u(z)-u(\bar z)|}{\dpd(z,\bar z)^{\ac}}=0.
\end{equation}
\end{theorem}

\begin{proof}

Suppose, by contradiction, that (\ref{eq:uniform-little-pub}) fails. Then, there exist \(\delta>0\), \(s_n\downarrow0\),
\(u_n\in\mathcal U\), and
\(z_n^0,\widehat z_n^0\in G_0\) such that $0<d_p(z_n^0,\widehat z_n^0)\le s_n$
and $|u_n(z_n^0)-u_n(\widehat z_n^0)|
\ge
\delta\,d_p(z_n^0,\widehat z_n^0)^{\alpha_c}.$  Then, we choose the nested sets
in \eqref{eq:nested-sets-pub} with the given \(G\Subset G_0\), and
apply Lemma~\ref{lem:bad-scale-selection-pub}.  Choose $0<\varepsilon<
 \min\{\delta,1,\lambda_*^{1/(\gamma-1)}\},$
The first rescaling sequence \eqref{eq:first-rescaling-pub} then has fixed
coefficient
\(\lambda=\varepsilon^{\gamma-1}\le\lambda_*\).  By using
\eqref{eq:first-local-holder-pub} and
Lemma~\ref{lem:data-disappearance-pub},  all conditions in Proposition \ref{prop:universal-energy-pub} hold. Then, we apply
Proposition~\ref{prop:first-compactness-pub}  to get, after extraction, $v_n\to v$ locally uniformly and strongly in $W^{2,1}_{\qc,\mathrm{loc}},$
and $Dv_n\to Dv$
\text{ in }$L^{\gamma\qc}_{\mathrm{loc}}.$
Using \eqref{eq:first-data-vanish-pub}, we pass to the limit in
\eqref{eq:first-rescaled-equation-pub} and obtain
\begin{equation}\label{eq:first-limit-equation-pub}
 v_t-\Delta v+\varepsilon^{\gamma-1}|Dv|^\gamma=0
 \quad\text{in }\R^d\times\R.
\end{equation}
Moreover,% For every fixed pair of points in \(\R^d\times\R\), the corresponding
%points belong to \(\Omega_n\) for large \(n\).  Hence
\eqref{eq:first-global-holder-pub} passes to the limit and gives $ [v]_{C^{\ac,\ac/2}(\R^d\times\R)}
 \le H_{G_2}/\varepsilon.$
The points \(\zeta_n\) is defined in \eqref{eq:first-normalization-pub}.  After extraction,
\(\zeta_n\to\zeta\), with \(\dpd(\zeta,0)=1\), and local uniform
convergence gives $v(0,0)=0$, $ |v(\zeta)|\ge1.$
Thus \(v\) is non-constant.  This contradicts
Lemma~\ref{lem:endpoint-liouville-pub}, applied to
\eqref{eq:first-limit-equation-pub}.  Hence, the desired conclusion follows.
\end{proof}

\section{Maximal regularity: proof of Theorem \ref{thm:main-pub}}
\label{sec:absorption-pub}
In this section, we prove our main result,
Theorem~\ref{thm:main-pub}. We begin by giving a  maximal
\(L^p\)-regularity estimate for the heat equation used later on as follows.
\begin{lemma} 
\label{lem:time-local-CZ-pub}
{Let $1<p<\infty$.} Let \(I_r\Subset I_R\Subset(0,T)\), \(0\le r<R\le1\), be intervals
such that
\(
 \operatorname{dist}(I_r,\partial I_R)\ge c_0(R-r).
\)
If \(u\in W^{2,1}_p(\T^d\times I_R)\), then
\begin{equation*}
 \|u_t\|_{L^p(\T^d\times I_r)}
 +\|D^2u\|_{L^p(\T^d\times I_r)}
 \le C\left(
 \|u_t-\Delta u\|_{L^p(\T^d\times I_R)}
 +(R-r)^{-2}\|u\|_{L^p(\T^d\times I_R)}
 \right).
\end{equation*}
{Here \(C=C(d,p,T,c_0)\) is independent of \(r\), \(R\) and \(u\).}
\end{lemma}

\begin{proof}
Choose \(\chi\in C_c^\infty(I_R)\) such that \(\chi=1\) on \(I_r\)
and \(\|\chi'\|_\infty\le C(R-r)^{-1}\).  Extend \(\chi u\) by zero
from the lower endpoint of \(I_R\), and apply  linear parabolic
{maximal regularity for the heat equation with zero initial data on
the torus; see \cite[Chapter~IV]{LadyzhenskayaSolonnikovUraltseva1968}.}  Since
\[
 (\chi u)_t-\Delta(\chi u)
 =\chi(u_t-\Delta u)+\chi'u,
\]
and \((R-r)^{-1}\le(R-r)^{-2}\), the claimed estimate follows.
\end{proof}

We now prove Theorem \ref{thm:main-pub}.
\begin{proof}[Proof of Theorem~\ref{thm:main-pub}]
Fix \(0<\tau<T/2\), and define, for \(0\le r\le1\),
\begin{equation*}
 I_r:=\left(\tau-\frac{r\tau}{2},
             T-\tau+\frac{r\tau}{2}\right).
\end{equation*}
Then \(I_0=(\tau,T-\tau)\),
\(I_1=(\tau/2,T-\tau/2)\) and $ \operatorname{dist}(I_r,\partial I_R)
 =\frac\tau2(R-r)$ for  $0\le r<R\le1$.
Set
\begin{equation*}
 X(r):=
 \|u_t\|_{L^{\qc}(\T^d\times I_r)}
 +\|D^2u\|_{L^{\qc}(\T^d\times I_r)}.
\end{equation*}
Let \(\sigma>0\), then by using Theorem~\ref{thm:uniform-little-pub}, we find there is
an \(r_0>0\), independent of \(u\in\U\), such that
\begin{equation*}
 \operatorname*{ess\,sup}_{t\in I_1}
 \sup_{0<\operatorname{dist}(x,y)\le4r_0}
 \frac{|u(x,t)-u(y,t)|}{\operatorname{dist}(x,y)^{\ac}}
 \le\sigma.
\end{equation*}
If necessary, we decrease \(r_0\) to ensure that  \(4r_0\) is
smaller than the injectivity radius of \(\T^d\).  By applying 
Lemma~\ref{lem:integrated-GN-pub}, on each \(I_r\) gives
\begin{equation}\label{eq:final-GN-pub}
 \bigl\||Du|^\gamma\bigr\|_{L^{\qc}(\T^d\times I_r)}
 \le C\sigma^{\gamma-1}
       \|D^2u\|_{L^{\qc}(\T^d\times I_r)}
      +C_{r_0,T}\sigma^\gamma,
\end{equation}
where $C_{r_0,T}$ is a positive constant.  Moreover, for \(0\le r<R\le1\), we apply
Lemma~\ref{lem:time-local-CZ-pub} to
$u_t-\Delta u=f_u-|Du|^\gamma.$
Using \eqref{eq:uniform-Linfty-pub}, boundedness of \(\F\) and
\eqref{eq:final-GN-pub} on \(I_R\), we obtain
\begin{equation}\label{eq:final-hole-pub}
 X(r)\le C\sigma^{\gamma-1}X(R)
 +C_{\sigma,r_0,T,\tau,K,
       \sup_{f\in\F}\|f\|_{L^{\qc}(Q_T)}}(R-r)^{-2}.
\end{equation}
Choose \(\sigma\) so small that
\(C\sigma^{\gamma-1}\le1/8\).  Set \(r_j=1-2^{-j}\), then the  Iteration of
\eqref{eq:final-hole-pub} gives
\begin{equation}\label{eq:final-iteration-pub}
 X(0)\le8^{-m}X(r_m)
 +C\sum_{j=0}^{m-1}8^{-j}(r_{j+1}-r_j)^{-2}.
\end{equation}
Noting that \(r_{j+1}-r_j=2^{-j-1}\), {the series is bounded by  \(C\sum_j2^{-j}\) for some constant $C>0$.  Moreover,
\(8^{-m}X(r_m)\le8^{-m}X(1)\to0\) since each strong solution has a
finite \(W^{2,1}_{\qc}(Q_T)\) norm.  Hence, letting \(m\to\infty\) in}
\eqref{eq:final-iteration-pub} yields
\begin{equation}\label{eq:final-X-bound-pub}
 \|u_t\|_{L^{\qc}(Q_\tau)}
 +\|D^2u\|_{L^{\qc}(Q_\tau)}\le C_\tau.
\end{equation}
Finally, \eqref{eq:final-GN-pub} on \(I_0\), together with
\eqref{eq:final-X-bound-pub}, gives
\[
 \bigl\||Du|^\gamma\bigr\|_{L^{\qc}(Q_\tau)}\le C_\tau.
\]
This completes the proof of \eqref{eq:main-estimate-pub}.
\end{proof}

\section{Existence: proof of Theorem \ref{thm:critical-existence-pub}}\label{sect5}
In this section, we prove Theorem~\ref{thm:critical-existence-pub}
by combining the maximal regularity estimate of
Theorem~\ref{thm:main-pub} with Proposition~\ref{prop:universal-energy-pub}.   We first establish a uniform \(L^\infty\) estimate for families of
classical solutions to \eqref{eq:main-equation-pub}.
\begin{lemma}
\label{lem:regularized-uniform-bounds-pub}
Let $\mathcal G\subset L^{\qc}(Q_T)$ be bounded and uniformly
equi-integrable, and let $\mathcal V\subset C(\T^d)$ be bounded and
equicontinuous.  Then, every classical solution of
\begin{equation*}
 \begin{cases}
 u_t-\Delta u+|Du|^\gamma=g&\text{in }Q_T,\\
 u(\cdot,0)=v&\text{on }\T^d,
 \end{cases}
 \qquad (g,v)\in\mathcal G\times\mathcal V,
\end{equation*}
satisfies
\begin{equation}\label{eq:regularized-Linfty-pub}
 \|u\|_{L^\infty(Q_T)}\le C,
\end{equation}
where $C$ is independent of $u$, $g$ and $v$.  Moreover,
\begin{equation}\label{eq:regularized-initial-modulus-pub}
 \lim_{t\downarrow0}
 \sup_{(g,v)\in\mathcal G\times\mathcal V}
 \|u(\cdot,t)-v\|_{L^\infty(\T^d)}=0.
\end{equation}
\end{lemma}

\begin{proof}
Write $q=\qc$ and $q'=q/(q-1)$.  Since $\gamma>2$,
\begin{equation}\label{eq:qc-above-heat-threshold-pub}
 q=\frac{(d+2)(\gamma-1)}{\gamma}>\frac{d+2}{2}.
\end{equation}
The inequality $u_t-\Delta u\le g$ and comparison with the heat equation
yield
\begin{equation}\label{eq:regularized-upper-comparison-pub}
 u(\cdot,t)
 \le e^{t\Delta}v+
 \int_0^t e^{(t-s)\Delta}g^+(\cdot,s)\,ds.
\end{equation}
By using the periodic heat-kernel estimate and
\eqref{eq:qc-above-heat-threshold-pub}, one has
\begin{equation}\label{eq:heat-potential-small-pub}
 \left\|\int_0^t e^{(t-s)\Delta}g^+(\cdot,s)\,ds
 \right\|_{L^\infty(\T^d)}
 \le C t^{1-\frac{d+2}{2q}}\|g\|_{L^q(Q_t)},
\end{equation}
where $q=q_c.$  It now  gives a uniform upper bound.

We prove the corresponding lower estimates by duality.  Fix
$t\in(0,T)$ and a smooth probability density $\rho_t$ on $\T^d$.  Put $b:=D_p(|Du|^\gamma)=\gamma|Du|^{\gamma-2}Du$
and let $\rho$ solve
\begin{equation}\label{eq:regularized-adjoint-pub}
 \begin{cases}
 -\rho_s-\Delta\rho{-}\operatorname{div}(b\rho)=0
     &\text{in }\T^d\times(0,t),\\
 \rho(\cdot,t)=\rho_t&\text{on }\T^d.
 \end{cases}
\end{equation}
Then $\rho\ge0$ and $\int_{\T^d}\rho(\cdot,s)=1$.  Set
\begin{align}\label{action-A}
 A:=\int_0^t\!\int_{\T^d}|b|^{\gamma'}\rho\,dx\,ds.
\end{align}
At $q=(d+2)/\gamma'$, the endpoint adjoint estimate
\cite[Corollary~2.3]{CirantGoffiParabolic2021} reads
\begin{equation}\label{eq:endpoint-adjoint-bound-pub}
 \|\rho\|_{L^{q'}(Q_t)}\le C(1+A),
\end{equation}
with $C=C(d,\gamma,T)$ independent of $\rho_t$.  The Legendre
transform of $H(p)=|p|^\gamma$ is
\[
 L(a)=c_\gamma|a|^{\gamma'},
 \qquad c_\gamma=(\gamma-1)\gamma^{-\gamma'},
\]
{Using $u_s=\Delta u-|Du|^\gamma+g$ and
$\rho_s=-\Delta\rho-\operatorname{div}(b\rho)$, we integrate by parts
on the torus to obtain
\[
\frac{d}{ds}\int_{\T^d}u\rho\,dx
=\int_{\T^d}\bigl(g-|Du|^\gamma+b\cdot Du\bigr)\rho\,dx
=\int_{\T^d}\bigl(g+(\gamma-1)|Du|^\gamma\bigr)\rho\,dx.
\]
Since $(\gamma-1)|Du|^\gamma=c_\gamma|b|^{\gamma'}$, integration in time yields}

\begin{equation}\label{eq:regularized-duality-identity-pub}
 \int_{\T^d}u(x,t)\rho_t(x)\,dx
 =\int_{\T^d}v(x)\rho(x,0)\,dx
  +\int_0^t\!\int_{\T^d}g\rho\,dx\,ds+c_\gamma A.
\end{equation}

Uniform equi-integrability implies the uniform tail condition
\begin{equation}\label{eq:regularized-tail-pub}
 \lim_{h\to\infty}\sup_{g\in\mathcal G}
 \|g\mathbf1_{\{|g|>h\}}\|_{L^q(Q_T)}=0.
\end{equation}
Choose $h$ so large that the last supremum is at most
$c_\gamma/(2C)$, where $C$ is the constant in
\eqref{eq:endpoint-adjoint-bound-pub}.  H\"older's inequality gives
\begin{align}
 \int_0^t\!\int_{\T^d}g\rho\,dx\,ds
 &\ge-h t-
 \|g\mathbf1_{\{|g|>h\}}\|_{L^q(Q_t)}
 \|\rho\|_{L^{q'}(Q_t)}\notag\\
 &\ge-hT-\frac{c_\gamma}{2}A-C.
 \label{eq:regularized-source-lower-pub}
\end{align}
Equations \eqref{eq:regularized-duality-identity-pub} and
\eqref{eq:regularized-source-lower-pub}, followed by approximation of a
Dirac mass with $\rho_t$, give a uniform lower bound for $u$.  Together
with \eqref{eq:regularized-upper-comparison-pub}, this proves
\eqref{eq:regularized-Linfty-pub}.

It remains to prove the lower initial modulus.  Given $\eta>0$, uniform
equicontinuity of $\mathcal V$ provides smooth functions $\phi_v$ such
that
\begin{equation*}
 \sup_{v\in\mathcal V}\|v-\phi_v\|_{L^\infty(\mathbb T^d)}\le\eta,
 \qquad
 \sup_{v\in\mathcal V}\|D\phi_v\|_{L^\infty(\mathbb T^d)}\le C_\eta.
\end{equation*}
{Fix $x\in\T^d$. We claim that the following estimate holds
for the adjoint solution with terminal datum $\delta_x$, with
$A$ denoting the corresponding action given in \eqref{action-A}.}
\begin{equation}\label{eq:adjoint-displacement-pub}
\int_{\T^d}\operatorname{dist}(x,y)\rho(y,0)\,dy
\le C\sqrt t+t^{1/\gamma}A^{1/\gamma'}.
\end{equation}

{To prove this, we first tackle an arbitrary smooth terminal
probability density $\rho_t$, and set $\rho$ and $A$ for its
adjoint solution and action.}
Set $d_x(y)=\operatorname{dist}(x,y)$, where the distance is taken
on $\T^d$.
Let $d_{x,\varepsilon}$ be a periodic mollification of $d_x$.
Since $d_x$ is $1$-Lipschitz, we have
\begin{equation}\label{above-bound}
\begin{aligned}
\|d_{x,\varepsilon}-d_x\|_{L^\infty(\T^d)}
&\le C\varepsilon,
\qquad
\|Dd_{x,\varepsilon}\|_{L^\infty(\T^d)}\le1,\\
\|\Delta d_{x,\varepsilon}\|_{L^\infty(\T^d)}
&\le C\varepsilon^{-1}.
\end{aligned}
\end{equation}
Testing \eqref{eq:regularized-adjoint-pub} with
$d_{x,\varepsilon}$ and integrating by parts gives
\begin{align*}
\int_{\T^d}d_{x,\varepsilon}(y)\rho(y,0)\,dy
&=\int_{\T^d}d_{x,\varepsilon}(y)\rho_t(y)\,dy\\
&\quad+\int_0^t\!\int_{\T^d}
\bigl(\Delta d_{x,\varepsilon}
-b\cdot Dd_{x,\varepsilon}\bigr)\rho\,dy\,ds.
\end{align*}
Using \eqref{above-bound} and
$\int_{\T^d}\rho(y,s)\,dy=1$, we obtain
\begin{align}\label{adjoint-dx-rho-dy}
\int_{\T^d}d_x(y)\rho(y,0)\,dy
&\le \int_{\T^d}d_x(y)\rho_t(y)\,dy
+C\left(\varepsilon+\frac{t}{\varepsilon}\right)
+\int_0^t\!\int_{\T^d}|b|\rho\,dy\,ds
\nonumber\\
&\le \int_{\T^d}d_x(y)\rho_t(y)\,dy
+C\left(\varepsilon+\frac{t}{\varepsilon}\right)
+t^{1/\gamma}A^{1/\gamma'},
\end{align}
where the last step follows from H\"older's inequality.
Choose $\varepsilon=\sqrt t$.

Now take smooth terminal probability densities
$\rho_t^{(j)}\rightharpoonup\delta_x$, and denote the corresponding
adjoint solutions and actions by $\rho^{(j)}$ and $A_j$.
For the fixed classical solution, the drift $b$ is bounded and
continuous. Standard stability of the linear adjoint equation
therefore gives
$\rho^{(j)}(\cdot,0)\rightharpoonup\rho(\cdot,0)$ and $A_j\to A$,
where $\rho$ has terminal datum $\delta_x$.
Since $\int_{\T^d}d_x(y)\rho_t^{(j)}(y)\,dy\rightarrow0,$
passing to the limit in \eqref{adjoint-dx-rho-dy} proves
\eqref{eq:adjoint-displacement-pub}.
The bound \eqref{eq:endpoint-adjoint-bound-pub} also yields weak
compactness of $\rho^{(j)}$ in $L^{q'}(Q_t)$, so the adjoint
estimate and the duality identity
\eqref{eq:regularized-duality-identity-pub} remain valid in this
limit, with terminal term $u(x,t)$.

Consequently,
\begin{equation}\label{eq:initial-term-lower-pub}
\int_{\T^d}v(y)\rho(y,0)\,dy
\ge v(x)-2\eta-C_\eta
\bigl(C\sqrt t+t^{1/\gamma}A^{1/\gamma'}\bigr).
\end{equation}
Given $\delta>0$, by \eqref{eq:regularized-tail-pub},
we can choose $h=h(\delta)$ sufficiently large that
\[
\sup_{g\in\mathcal G}
\|g\mathbf1_{\{|g|>h\}}\|_{L^{q_c}(Q_T)}
\le\delta.
\]
By using Young's inequality, one has 
\[
C_\eta t^{1/\gamma}A^{1/\gamma'}
\le \frac{c_\gamma}{2}A+C_{\eta,\gamma}t.
\]
Combining this with \eqref{eq:endpoint-adjoint-bound-pub},
\eqref{eq:regularized-duality-identity-pub} and
\eqref{eq:initial-term-lower-pub}, we obtain
\begin{equation}\label{eq:lower-initial-modulus-pub}
u(x,t)-v(x)
\ge-2\eta-C_\eta\sqrt t-C_{\eta,\gamma}t
-h(\delta)t-C\delta
+\bigl(c_\gamma/2-C\delta\bigr)A.
\end{equation}
Choose $\delta$ small enough that the coefficient of $A$ is
nonnegative, and omit this term from the lower bound.
For fixed $\delta$ and $\eta$, all terms on the right-hand side
of \eqref{eq:lower-initial-modulus-pub} containing $t$ vanish
as $t\downarrow0$.
Then letting $\delta\downarrow0$ and $\eta\downarrow0$ proves
the lower estimate needed for
\eqref{eq:regularized-initial-modulus-pub}.
The upper bound follows from
\eqref{eq:regularized-upper-comparison-pub}--
\eqref{eq:heat-potential-small-pub}, together with
\[
\sup_{v\in\mathcal V}
\|e^{t\Delta}v-v\|_{L^\infty({\T^d})}
\rightarrow0
\qquad\text{as }t\downarrow0,
\]
which follows from the uniform equicontinuity of $\mathcal V$.
Combining the two bounds proves
\eqref{eq:regularized-initial-modulus-pub}.
\end{proof}

Now, we are ready to prove Theorem \ref{thm:critical-existence-pub}:
\begin{proof}[Proof of Theorem~\ref{thm:critical-existence-pub}]
 {We extend $f$ by zero for $t\notin(0,T)$. Standard space-time
mollification, performed periodically in space, yields
$f_n\in C^\infty(\T^d\times[0,T])$.}
Let $u_{0,n}$ be smooth 
periodic mollifications of $u_0$ such that
\begin{equation}\label{eq:critical-data-approximation-pub}
 f_n\to f\quad\text{in }L^{\qc}(Q_T),
 \qquad
 u_{0,n}\to u_0\quad\text{uniformly on }\T^d.
\end{equation}
For every $n$, we shall apply the continuation argument to show that the regularized problem 
\begin{equation}\label{eq:smooth-approximate-equation-pub}
 \begin{cases}
 (u_n)_t-\Delta u_n+|Du_n|^\gamma=f_n&\text{in }Q_T,\\
 u_n(\cdot,0)=u_{0,n}&\text{on }\T^d
 \end{cases}
\end{equation}
has a unique classical solution on $[0,T]$.  Thanks to the standard local
semilinear parabolic theory, one has there exists  a classical solution for the smooth
periodic data in a short time interval. On its maximal interval of existence, set
$w_n=|Du_n|^2/2$. Differentiating the equation in \eqref{eq:smooth-approximate-equation-pub}, one finds
\[
(w_n)_t-\Delta w_n+
\gamma|Du_n|^{\gamma-2}Du_n\cdot Dw_n
=Df_n\cdot Du_n-|D^2u_n|^2.
\]
Applying the maximum principle to $(2w_n+\varepsilon)^{1/2}$ and then
letting $\varepsilon\downarrow0$, we obtain that 
\[
\|Du_n(\cdot,t)\|_{L^\infty(\mathbb T^d)}
\le\|Du_{0,n}\|_{L^\infty(\mathbb T^d)}+
\int_0^t\|Df_n(\cdot,s)\|_{L^\infty(\mathbb T^d)}\,ds.
\]
Also, we use comparison principle to get 
$\|u_n(\cdot,t)\|_{L^\infty(\mathbb T^d)}\le\|u_{0,n}\|_{L^\infty(\mathbb T^d)}+
\int_0^t\|f_n(\cdot,s)\|_{L^\infty(\mathbb T^d)}\,ds$.
Since the bounds on $u_n$ and $Du_n$ remain finite up to time $T$ and
$H(p)=|p|^\gamma$ is locally Lipschitz, the classical solution
extends to $[0,T]$ by the standard continuation criterion In addition, by using the maximum
principle, one obtains the uniqueness of classical solutions. Invoking the standard linear parabolic regularity with
smooth initial data, one arrives at  $u_n\in W^{2,1}_{\qc}(Q_T)$, as required
by Theorem~\ref{thm:main-pub}.

Set  $\mathcal F_0:=\{f\}\cup\{f_n:n\in\mathbb N\}.$
The strong convergence in \eqref{eq:critical-data-approximation-pub}
implies that $\mathcal F_0$ is bounded and uniformly equi-integrable in
$L^{\qc}(Q_T)$.  Indeed,
\[
 \bigl\||f_n|^{\qc}-|f|^{\qc}\bigr\|_{L^1(Q_T)}
 \le C_{\qc}
 \bigl(\|f_n\|_{L^{\qc}(Q_T)}^{\qc-1}
       +\|f\|_{L^{\qc}(Q_T)}^{\qc-1}\bigr)
 \|f_n-f\|_{L^{\qc}(Q_T)}
 \rightarrow0.
\]
Hence, for all sufficiently large $n$, the integral of $|f_n|^{\qc}$ over
a set of small measure is uniformly small since $f_n\to f$ in
$L^{\qc}(Q_T)$.   Therefore, $\mathcal F_0$
is uniformly equi-integrable.

The family $\{u_{0,n}:n\in\mathbb N\}$ is uniformly bounded and
equicontinuous.  Lemma~\ref{lem:regularized-uniform-bounds-pub} therefore
gives
\begin{equation}\label{eq:approximate-uniform-bound-pub}
 \sup_n\|u_n\|_{L^\infty(Q_T)}<\infty
\end{equation}
and
\begin{equation}\label{eq:uniform-initial-modulus-pub}
 \lim_{t\downarrow0}
 \sup_n\|u_n(\cdot,t)-u_{0,n}\|_{L^\infty(\T^d)}=0.
\end{equation}

Invoking Theorem~\ref{thm:main-pub}, we have for every $0<\tau<T/2$,
\begin{equation*}
 \sup_n\left(
 \|(u_n)_t\|_{L^{\qc}(Q_\tau)}
 +\|D^2u_n\|_{L^{\qc}(Q_\tau)}
 +\||Du_n|^\gamma\|_{L^{\qc}(Q_\tau)}
 \right)<\infty.
\end{equation*}
{For every compact set $K\Subset Q_T$, Theorem~\ref{thm:critical-holder-pub},}
together with \eqref{eq:approximate-uniform-bound-pub}, yields
$\sup_n\|u_n\|_{C^{\alpha_c,\alpha_c/2}(K)}<\infty$.
Thus, $(u_n)$ is uniformly bounded and equicontinuous on $K$.  By the
Arzel\`a--Ascoli theorem and a diagonal argument, there exists a
continuous function $u$ such that
\begin{equation}\label{eq:approximate-local-uniform-pub}
 u_n\rightarrow u
 \quad\text{locally uniformly in }Q_T.
\end{equation}

We next rule out concentration of the critical energy.  We claim that, for
every $0<\tau<T/2$,
\begin{equation}\label{eq:original-energy-nonconcentration-pub}
 \lim_{r\downarrow0}
 \sup_n\sup_{z_0\in Q_\tau}
 \Energy\bigl(u_n;Q_r(z_0)\bigr)=0,
\end{equation}
where $ \Energy$ is given in \eqref{eq:critical-energy-pub} and $r$ is chosen such that
$Q_{4r}(z_0)\subset Q_{\tau/2}$.  Suppose, to the contrary, that there are
$\eta>0$, $r_k\downarrow0$, indices $n_k$, and points
$z_k=(x_k,t_k)\in Q_\tau$ such that
\begin{equation}\label{eq:original-energy-concentration-pub}
 \Energy\bigl(u_{n_k};Q_{r_k}(z_k)\bigr)\ge\eta.
\end{equation}
After extraction, we may assume that $n_k\to\infty$; otherwise the
absolute continuity of the energy of one fixed strong solution contradicts
\eqref{eq:original-energy-concentration-pub}.

Define
\begin{align}
 v_k(y,s)
 &:=\frac{u_{n_k}(x_k+r_ky,t_k+r_k^2s)-u_{n_k}(x_k,t_k)}
          {r_k^{\alpha_c}},
 \label{eq:existence-energy-rescaling-pub}\\
 g_k(y,s)
 &:=r_k^{\gamma'}
 f_{n_k}(x_k+r_ky,t_k+r_k^2s).
 \label{eq:existence-data-rescaling-pub}
\end{align}
On the rescaled domains, which exhaust
$\mathbb R^d\times\mathbb R$, these functions satisfy
\begin{equation*}
 (v_k)_s-\Delta v_k+|Dv_k|^\gamma=g_k.
\end{equation*}
Theorem~\ref{thm:uniform-little-pub} implies that, for every fixed $L>0$,
\begin{equation}\label{eq:existence-rescaled-holder-zero-pub}
 [v_k]_{C^{\alpha_c,\alpha_c/2}(Q_L)}\rightarrow0.
\end{equation}
Indeed, the points corresponding to $Q_L$ in the original variables have
parabolic distance at most $C_Lr_k$.  Uniform equi-integrability of
$\mathcal F_0$ and the identity $\gamma'\qc=d+2$ also give the
scale-uniform disappearance condition \eqref{eq:admissible-data-pub} for
$(g_k)$.  Hence Proposition~\ref{prop:universal-energy-pub}, applied with
$\lambda_k=1$, yields
\begin{equation}\label{eq:existence-rescaled-energy-bound-pub}
 \sup_k\Energy(v_k;Q_L)<\infty
 \qquad\text{for every fixed }L>0.
\end{equation}

Applying Lemma~\ref{lem:GN-pub} 
and using a finite covering of $Q_2$, we obtain from \eqref{eq:existence-rescaled-holder-zero-pub} and
\eqref{eq:existence-rescaled-energy-bound-pub} that 
\begin{equation}\label{eq:existence-gradient-zero-pub}
 \||Dv_k|^\gamma\|_{L^{\qc}(Q_2)}\rightarrow0.
\end{equation}
Moreover, $g_k\to0$ in $L^{\qc}(Q_2)$ and $v_k\to0$ uniformly on $Q_2$.
The standard local parabolic estimate applied to $(v_k)_s-\Delta v_k=g_k-|Dv_k|^\gamma$
therefore gives
\begin{equation}\label{eq:existence-linear-energy-zero-pub}
 \|(v_k)_s\|_{L^{\qc}(Q_1)}
 +\|D^2v_k\|_{L^{\qc}(Q_1)}\rightarrow0.
\end{equation}
  \eqref{eq:existence-gradient-zero-pub} and
\eqref{eq:existence-linear-energy-zero-pub} imply
$\Energy(v_k;Q_1)\to0$.  By \eqref{eq:existence-energy-rescaling-pub}---\eqref{eq:existence-data-rescaling-pub}, one finds $\Energy(v_k;Q_1)
 =\Energy\bigl(u_{n_k};Q_{r_k}(z_k)\bigr),$
contradicting \eqref{eq:original-energy-concentration-pub}.  This proves
\eqref{eq:original-energy-nonconcentration-pub}.

We now prove strong compactness.  For $n,m\in\mathbb N$, let
$Z_{nm}:=u_n-u_m$ and define
\[
 \mathbf B_{nm}
 :=\gamma\int_0^1
 |Du_m+\theta(Du_n-Du_m)|^{\gamma-2}
 \bigl(Du_m+\theta(Du_n-Du_m)\bigr)\,d\theta.
\]
Subtracting the equations for $u_n$ and $u_m$ gives
\begin{equation}\label{eq:existence-difference-equation-pub}
 (Z_{nm})_t-\Delta Z_{nm}
 +\mathbf B_{nm}\cdot DZ_{nm}=f_n-f_m.
\end{equation}
Since $(\gamma-1)(d+2)=\gamma\qc,$
we have, on every parabolic cylinder $Q_{2r}(z_0)$,
\begin{align}
 \|\mathbf B_{nm}\|_{L^{d+2}(Q_{2r}(z_0))}^{d+2}
 &\le C_{d,\gamma}
 \int_{Q_{2r}(z_0)}
 \bigl(|Du_n|^{\gamma\qc}+|Du_m|^{\gamma\qc}\bigr).
 \label{eq:existence-drift-energy-pub}
\end{align}
By \eqref{eq:original-energy-nonconcentration-pub}, we may choose $r>0$
so small that the right-hand side of
\eqref{eq:existence-drift-energy-pub} is below
$\kappa_{\qc}^{d+2}$, uniformly in $n,m$ and in
$z_0\in Q_\tau$.  The rescaled form of
Lemma~\ref{lem:small-drift-pub}(i), applied to
\eqref{eq:existence-difference-equation-pub}, gives
\begin{align}
 &\|(Z_{nm})_t\|_{L^{\qc}(Q_r(z_0))}
 +\|D^2Z_{nm}\|_{L^{\qc}(Q_r(z_0))}
 \notag\\
 &\qquad\le C_r\left(
 \|f_n-f_m\|_{L^{\qc}(Q_{2r}(z_0))}
 +\|Z_{nm}\|_{L^{\qc}(Q_{2r}(z_0))}
 \right).
 \label{eq:existence-local-cauchy-pub}
\end{align}
The first term in \eqref{eq:existence-local-cauchy-pub} on the right tends to zero by
\eqref{eq:critical-data-approximation-pub}, and the second tends to zero by
\eqref{eq:approximate-local-uniform-pub}.  A finite covering of $Q_\tau$
by such cylinders shows that $(u_n)$ is Cauchy in
$W^{2,1}_{\qc}(Q_\tau)$.  Consequently,
 $u_n\rightarrow u
 \quad\text{strongly in }W^{2,1}_{\qc}(Q_\tau).$
The continuous  embedding then yields
 $Du_n\rightarrow Du$ \text{strongly in }$L^{\gamma\qc}(Q_\tau)$.

Finally, the pointwise inequality $\bigl||p|^\gamma-|q|^\gamma\bigr|
 \le C_\gamma\bigl(|p|^{\gamma-1}+|q|^{\gamma-1}\bigr)|p-q|$
and H\"older's inequality yield
\begin{align*}
 \||Du_n|^\gamma-|Du|^\gamma\|_{L^{\qc}(Q_\tau)}
 &\le C_\gamma
 \bigl(
 \|Du_n\|_{L^{\gamma\qc}(Q_\tau)}^{\gamma-1}
 +\|Du\|_{L^{\gamma\qc}(Q_\tau)}^{\gamma-1}
 \bigr)
 \|Du_n-Du\|_{L^{\gamma\qc}(Q_\tau)}
 \rightarrow0.
\end{align*}
We may therefore pass to the limit in
\eqref{eq:smooth-approximate-equation-pub}, proving
\eqref{eq:critical-limit-equation-pub}.
{The bound \eqref{eq:approximate-uniform-bound-pub} passes to the limit,
so $u\in L^\infty(Q_T)$.}

It remains to identify the initial trace.  For every fixed $t>0$, local
uniform convergence gives $u_n(\cdot,t)\to u(\cdot,t)$ uniformly on
$\T^d$.  Hence
\[
 \|u(\cdot,t)-u_0\|_{L^\infty(\T^d)}
 \le\limsup_{n\to\infty}
 \left(
 \|u_n(\cdot,t)-u_{0,n}\|_{L^\infty(\T^d)}
 +\|u_{0,n}-u_0\|_{L^\infty(\T^d)}
 \right).
\]
Letting $t\downarrow0$ and using
\eqref{eq:critical-data-approximation-pub}--
\eqref{eq:uniform-initial-modulus-pub} proves
the initial condition in \eqref{eq:critical-limit-equation-pub} and
completes the proof.
\end{proof}

\bibliographystyle{plain}
\bibliography{ref}

\end{document}